\documentclass[12pt]{amsart}
\usepackage[latin1]{inputenc}
\usepackage[english]{babel}
\usepackage[pdftex]{graphicx}
\usepackage{wrapfig}
\usepackage{pdfsync}
\usepackage{epsfig}
\usepackage{epstopdf}
\usepackage{latexsym,amssymb,amsfonts,amsmath, amscd, euscript, MnSymbol}
\usepackage{enumerate}
\usepackage{verbatim}

\newcommand{\M}{\mathbb{M}}

\newcommand{\R}{\mathbb{R}}
\newcommand{\N}{\mathbb{N}}

\newcommand{\Q}{\mathbb{Q}}

\def\Power #1 { \powerset(#1) }
\def\Bidom #1 { {\mathfrak P} (#1) }
\newcommand{\sib}{sib}

\newcommand{\Max}{Max}

\newcommand{\age}{Age}
\newcommand{\Mon}{\mathbf M}
\newcommand{\Cog}{\mathcal Cog}
\newtheorem{definition}{{\bf Definition}}[section]
\newtheorem{theorem}[definition]{{\bf Theorem}}

\newtheorem{proposition}[definition]{\noindent {\bf Proposition}}
\newtheorem{lemma}[definition]{\noindent {\bf Lemma}}

\newtheorem{claim}[definition]{\noindent {\bf Claim}}

\newtheorem{example}[definition]{\noindent {\bf Example}}

\newtheorem{problem}[definition]{\noindent {\bf Problem}}

\newtheorem{conjecture}[definition]{\noindent {\bf Conjecture}}

\def\proofref #1 {{\noindent  {\bf Proof} (#1).}\ }

\def\endproof{\hfill {\kern 6pt\penalty 500
\raise -0pt\hbox{\vrule \vbox to5pt {\hrule width 5pt
\vfill\hrule}\vrule}}}
\def\centerpicture #1 by #2 (#3){\leavevmode
        \vbox to #2{
        \hrule width #1 height 0pt depth 0pt
        \vfill
        \special{pictfile #3}}}
\title[Siblings]{Siblings of countable cographs}

\author[G.Hahn]{Ge\v na  Hahn*} 
\address{Informatique et Recherche Op\'erationnelle, Universit\'e de
Montr\'eal, CP 6128  succursale Centre-Ville, Montr\'eal, Qu\'ebec,
H3C 3J7, Canada}
\email{hahn@iro.umontreal.ca} 
\thanks{The visit of the first author in November 2019 was partially supported by the Institut Camille Jordan of Lyon and   by NSERC of Canada Grant \# } 
\author[M.Pouzet]{Maurice Pouzet} \address{ICJ, Math\'ematiques,
Universit\'e Claude-Bernard Lyon1, 43 bd. 11 Novembre 1918, 69622
Villeurbanne Cedex, France and Mathematics \& Statistics, University of Calgary, Calgary, Alberta, T2N 1N4, Canada}

 \email{pouzet@univ-lyon1.fr }
\author[R.E. Woodrow]{Robert Woodrow**} \address {Mathematics \&
Statistics, University of Calgary, Calgary, Alberta, T2N 1N4, Canada }\thanks{***The stay of the  third author in November 2019 was supported by  LABEX MILYON (ANR-10-LABX-0070) of Universit\'e de Lyon within the program ``Investissements d'Avenir (ANR-11-IDEX-0007)" operated by the French National Research Agency (ANR)}
\email{woodrow@ucalgary.ca>}  

\date{\today}

\begin{document}

\dedicatory{\dagger Dedicated to the memory of Ivo G. Rosenberg}


%\begin{abstract} Two structures are \emph{equimorphic} if each embeds in the other. If they are infinite, these structures do not need to be isomorphic. Two conjectures relating the notions of equimorphy and isomorphy  for infinite structures are the substance of these notes.  I show that an extension of a conjecture of Bonato and Tardif to (undirected and loopless) graphs,  if true,  implies that a  conjecture of Thomass\'e on graphs is true too. 

\begin{abstract} We show that every countable cograph has either one or infinitely many siblings. This answers, very  partially, a conjecture of Thomass\'e. The main tools are the notion of well quasi ordering and   the correspondence between cographs and some labelled ordered trees. 
\end{abstract} 

\maketitle

%==========================================================================

%--------------------------
\section{Introduction} 
%--------------------------

\subsection{Thomass\'e conjecture} A \emph{relation}  $R$ is a pair $(V, \rho)$ where $\rho$ is a subset of $V^n$ for some non negative integer $n$; this integer is the \emph{arity} of $\rho$ (which is also called an $n$-ary relation).  A \emph{sibling} of a relation $R$ is any $R'$ such that $R$ and $R'$ are embeddable in each other.  In \cite{thomasse1} (p.2,  Conjecture 2), Thomass\'e made the following:
%More generally a \emph{relational structure} $R$ is a pair $(V, (\rho_i)_{i\in I})$ where each $\rho_i$ is a $n_i$-ary relation on $V$. 
%
\begin{conjecture} Every countable relation $R$ has  $1$, $\aleph_0$ or $2^{\aleph_0}$ siblings, these siblings counted  up to isomorphy. 
\end{conjecture}

A positive answer was given for chains in \cite{LPW}.
It is not even known if $R$ has one or infinitely many siblings, even if $R$ is a countable loopless undirected graph;  in fact, this is unsolved for countable trees  as we will see later with a conjecture of Bonato and Tardif.  

Graphs, and more generally binary structures, can be decomposed in simpler pieces via labelled trees, the bricks being the indecomposable structures. A binary relation $R:= (V, \rho)$ is \emph{indecomposable} if it has no  non-trivial module, alias "interval"; a \emph{module} of $R$ is any subset $A$ of $V$ such that for every $a,a'\in A$, $b\in V\setminus A$, the equalities 
$\rho(a,b)=\rho(a',b)$ and $\rho(b,a)=\rho(b,a')$ hold; $A$ is \emph{trivial} if $A=\emptyset$, $A = \{a\}$ for some $a\in V$, or $A=V$ (see \cite{ehrenfeucht1} for the general theory, \cite{harju-rozenberg}, \cite{ille-woodrow1, ille-woodrow2, boussairi-al} and \cite{courcelle-delhomme} for infinite binary structures). 

We  conjecture that Thomass\'e's conjecture reduces to the case of countable indecomposable structures. That is for graphs:

\begin{conjecture}  A countable graph $G$ has $1$, $\aleph_0$ or $2^{\aleph_0}$ siblings if every induced indecomposable subgraph has $1$, $\aleph_0$ or $2^{\aleph_0}$ siblings\end{conjecture}

This paper is a contribution in that direction. 

If our conjecture holds, then Thomass\'e's conjecture  must hold  for countable cographs. A \emph{cograph} is a graph with no induced subgraph isomorphic to a $P_4$, a path on four vertices. As it is well known, no induced subgraph of a cograph with more than two vertices can be indecomposable (see \cite{sumner} for finite graphs,  \cite{kelly} for infinite graphs).

We prove a weaker version of Thomass\'e's conjecture:
\begin{theorem} \label{thm:siblings-cographs}A countable cograph has  either one or infinitely many siblings.
\end{theorem}

With the continuum hypothesis (CH), this says that  a cograph with countably many vertices has one, $\aleph_0$  or $2^{\aleph_0}$ many siblings. We do hope to get rid of (CH) in a forthcoming publication.

\subsection{Connected siblings and the conjectures of Bonato, Bruhn, Diestel, Spr\"ussel and Tardif}

Indecomposable graphs with more than two vertices must be connected, hence a special consequence of our conjecture is the following fact, observed  in \cite{gagnon-hahn-woodrow}.

\begin{proposition}\label{prop:thomasse}

 If Thomasse's conjecture holds for connected graphs, it holds for all graphs. 

\end{proposition}
 \begin{proof}
 Suppose that Thomass\'e conjecture holds for connected graphs.  Let  $G$ be a countable  graph.  If  $G$ is connected, then  it has one, $\aleph_0$  or $2^{\aleph_0}$ many siblings. If $G$ is not connected, then  its  complement $G^{c}$ is connected and thus has one, $\aleph_0$  or $2^{\aleph_0}$ many siblings. Since their complements are siblings of $G$,   $G$  has one, $\aleph_0$  or $2^{\aleph_0}$ many siblings. 
\end{proof}

The study of the number of siblings  of a direct sum of connected graphs relates to a conjecture of Bonato and Tardif about trees, not ordered trees but connected graphs with no cycles.   
The \emph{tree alternative property} holds for  a tree $T$,  if either every tree  equimorphic to $T$ is  isomorphic to $T$ or there  are infinitely many pairwise non-isomorphic trees which are equimorphic to $T$. Bonato and Tardif \cite {Bo-Ta} conjectured that the tree alternative property holds for every tree and proved that it holds for rayless trees \cite{Bo-Ta}. Laflamme, Pouzet, Sauer  \cite{La-Po-Sa} proved that it holds for scattered trees (trees in which  no subdivision of the binary tree can be embedded). But they could not conclude in the case of the complete ternary tree with some leaves attached. As it turns out, induced subgraphs equimorphic to these trees  are connected; hence for such trees the tree alternative property amounts to the fact that every tree has one or infinitely many siblings.

The Bonato-Tardif conjecture was extended to (undirected loopless) graphs by Bonato et al,  (2011)\cite{Bo-al} in the following ways:
\begin{enumerate}
\item  For every connected graph $G$ the number $\sib_{conn}(G)$ of  connected graphs which are equimorphic to $G$ is $1$ or is infinite. 

\item For every  graph $G$ the number $\sib(G)$  of   graphs which are equimorphic to $G$ is $1$ or is infinite. 
 \end{enumerate}
 The second conjecture is a weakening  of Thomasse's conjecture restricted to  graphs. 
Both conjectures were proved true for  rayless graphs by  Bonato et al   (2011)\cite{Bo-al}.

Note  that the extension of  the  first conjecture of Bonato et  al  to binary relations is false. In fact it is false for undirected graphs with loops and for ordered sets. 

\begin{center}
\begin{figure}[ht]
\includegraphics[width=4in]{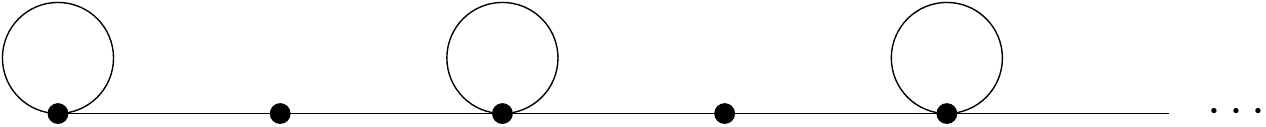}
\caption{A path with loops}
\includegraphics[width=4in]{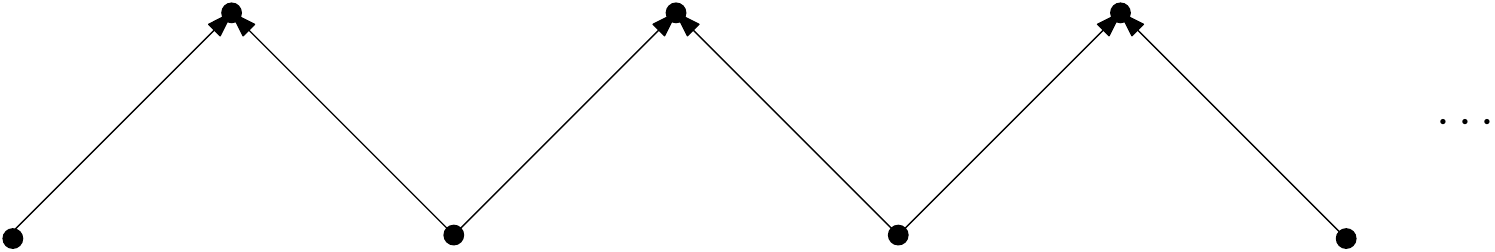}
 \caption{A fence}
 \end{figure}
 \label{Examples}
\end{center}

There is a straightforward relationship between 
the extension of the Bonato-Tardif conjecture to connected graphs and  the weakening of  Thomasse's conjecture  for graphs. 

To see this,  first observe that:
\begin{lemma}\label{lem:disconnected1} If some sibling of a connected graph $G$ is not connected, then $\sib(G)$ is infinite.
\end{lemma}
Indeed,  $G$ is equimorphic to the direct sum $G\oplus 1$. In this case, $G\oplus \overline K_n$, where $\overline K_n$ is an independent set of size $n$, and $n$ any positive integer, is equimorphic to $G$, hence $G$ has infinitely many siblings.
Now, we have:
\begin{proposition}\label{prop:BT-T} If the extension of the Bonato-Tardif conjecture to connected graphs is true  then the weakening  of   Thomasse's conjecture  for graphs is true. 
\end{proposition}  

\begin{proof}  According to Proposition \ref{prop:thomasse},   a graph   has one or infinitely many siblings provided that all connected graphs have one of infinitely many siblings. Let $G$ be  a connected graph. 
If all siblings of $G$ are connected, apply the extension of the Bonato-Tardif conjecture: $G$ has one or infinitely many siblings. If some sibling is not connected, apply Lemma \ref{lem:disconnected1}.\end{proof}

In  November 2016, M.H. Shekarriz sent  us  a paper in which he proves the same result.
 
We prove that the extension of the Bonato-Tardif conjecture to countable connected cographs holds:
\begin{theorem}\label{thm:connected-siblings}
A countable connected cograph has either one connected sibling or infinitely many connected siblings. 
\end{theorem}

Theorem \ref{thm:siblings-cographs} follows from this  and Proposition \ref{prop:BT-T}.

Apart from the one way infinite path, there are several countable connected graphs $G$ with  $\sib_{conn}(G)=1$ and $\sib(G)$  infinite. This is not the case with countable cographs. For a graph $G$ set $\sib_{c} (G)= \sib_{conn}(G)$ if $G$ is connected and $\sib_{c} (G)= \sib_{conn}(G^c)$ if $G$ is disconnected. We  prove that:

\begin{theorem}\label{thm:cograph- one-sibling} 
For a countable cograph, the following properties are equivalent
\begin{enumerate} [{(i)}] 
\item $\sib(G)= 1$; 
\item $\sib_c(G)=1$; 
\item $G$ is a finite lexicographic sum of cliques or independent sets.
\end{enumerate}
\end{theorem}

The proofs of these two results have two ingredients. One is the tree decomposition of a cograph. Each cograph can be represented uniquely, up to isomorphism by a special type of (ordered) labelled tree. This fact is a special case of a very general result  about tree decomposition of binary structures; a result which appears in \cite{courcelle-delhomme}, and based on   \cite {ehrenfeucht, ehrenfeucht1, harju-rozenberg} (a similar approach is in  \cite{ille-woodrow1, ille-woodrow2, boussairi-al}). Due to the importance of the fact mentionned above, we  give a detailed presentation and a proof  (Theorem \ref{thm:correspondence}) in the appendix.
With that fact,  constructions of non isomorphic cographs reduce to constructions of labelled trees. We reduce our counting of siblings to the case of trees with no least element, for which we prove that there are $2^{\aleph_0}$ siblings (Theorem \ref{label:chain/antichain}). This reduction is based on  the second ingredient. This is the notion of well quasi ordering . A quasi-ordered set $Q$ is \emph{well quasi-ordered} (w.q.o. for short) if every infinite sequence $q_0, \dots, q_n, \dots$ contains an infinite  increasing subsequence  $q_{n_0}\leq \cdots q_{n_k}\leq \cdots$. In particular, $Q$ is \emph{well founded}: every non-empty subset $A$ contains a minimal element (an element $a$ such that no $b<a$, i.e., $b\leq a$ and $a\not \leq b$) is in $A$. This property allows induction on the elements of $Q$.  An example is the collection $\Cog_{\leq \omega} $ of countable cographs quasi-ordered by embeddability. Indeed, according to a theorem of Thomass\'e \cite{thomasse}, $\Cog_{\leq \omega} $ is well quasi-ordered  by embeddability. The reduction to the case of   trees with no least element in  Theorem  \ref{thm:connected-siblings} relies on a quasi order slightly different from embeddability which turns to be   w.q.o.  in virtue of Thomass\'e's result. The counting of siblings in the case of these trees, which  is the most difficult part of the paper, relies on properties of countable labelled chains. Induction can be done if the collection of our countable labelled chains is w.q.o. In general, the collection of countable chains labelled by a w.q.o. is not necessarily w.q.o. A strengthening of this notion is needed, this  is the notion of better quasi ordering (in short b.q.o.),  invented by C.St.J.A. Nash-Williams \cite{nashwilliams}, as a tool  for proving that some posets are w.q.o. No expertise about  b.q.o. is needed in this paper,  but the reader must be aware that this notion  is unavoidable to prove that some posets are w.q.o. According to  Laver \cite{laver},   the collection $\mathcal Q^{\mathcal {D}_{\leq \omega}}$ of  countable chains labelled by a b.q.o $Q$ is b.q.o.(see p. 90 of  \cite{laver}). The set $Q$ of labels we need is the direct product of the collection $\Cog_{\leq \omega} $ of countable cographs and the $2$-element antichain.   Since Thomass\'e proved in fact that that  $\Cog_{\leq \omega} $ is b.q.o., $Q$ is b.q.o. so our collection of labelled chains  is w.q.o and we can do induction (see Subsection \ref{subsectionlabelled}). We conclude  with  some problems (see Section \ref{section:extension}).

Some results about the  conjectures above were presented  in \cite{pouzet} and also in \cite{gagnon-hahn-woodrow}.

\section{Ingredients}

\subsection{Elementary facts}
We start with easy facts. 
%\begin{lemma} \label{lem:disconnected siblings} If a connected graph and a disconnected graph are equimorphic these graphs have infinitely many disconnected siblings. 
%\end{lemma} 
%\begin{proof} Let $G$ connected and $G'$ disconnected satisfy the conditions of the lemma.
%Since $G$ is connected, $G$ embeds into some connected component of $G'$ hence $G\oplus 1$ embeds into $G'$.  It follows that  $G$ is equimorphic  to $G\oplus 1$  thus equimorphic to  $G\oplus \overline K_n$ for every integer $n$. The conclusion of the lemma follows.
%\end{proof}
\begin{lemma} \label{lem:onesiblingdirectsum}If a disconnected graph $G$ has finitely many non-trivial connected components and each connected component  has just one connected sibling then either $G$ has infinitely many disconnected siblings  or just one sibling. 
\end{lemma}
\begin{proof}
 Let $G$ satisfy the conditions of the lemma.  Either every component is trivial, in that case $G$ is an independent set and $\sib(G)=1$, or some component is not trivial. In that case, let $H_1,\dots,  H_k$ enumerate the non-trivial components. 
 
 \noindent {\bf Case 1}  There is a component $H$ and a non-trivial and connected graph $L$ such that $H\oplus L\leq H$. In that case, for each $n\geq 1$,  the direct sum of $G$ with the direct sum of $n$ copies of $L$ is a disconnected sibling of $G$, hence $G$ has infinitely many disconnected siblings.
 
 \noindent  {\bf Case 2}   Otherwise,   $H\oplus L\leq H$ implies that $L$ cannot be connected. In this case, if the set $T$ of trivial components is non empty and for some connected component $H$, $H\oplus \overline K_{\kappa}\leq H$, where $\kappa:= \vert T\vert$ then $G$ is equimorphic to the set $G'$ of non-trivial components, and also to $G'\oplus 1$ thus $G$ has infinitely many disconnected siblings. If $T$ is infinite or $H\oplus 1\not \leq H$ for each connected component $H$ then $\sib (G)= 1$. Indeed, suppose that $G' \subseteq G$ induces a sibling of $G$ via an embedding $\phi$.   Since each $H_i$  is connected the image under $\phi$ must be a contained in some component $H_{\varphi(i)}$ of $G$. Furthermore, the map $\varphi$ must be one to one, hence bijective (otherwise, by iterating it we will find $i\not = j$ such that $H_i\oplus H_j$ embeds into $H_j$ contradicting our hypothesis). We may suppose that $\varphi$ is the identity, so, if $T$ is infinite, infinitely many elements will stay in $T$ and if $T$ is finite, none will be mapped  into some non-trivial component, hence  $G'$ is isomorphic to $G$.  \end{proof}
% As there are only finitely many non-trivial components, let $H_1,\dots H_k$ enumerate those that are maximal with respect to embedding.  That is, if $H_i$ embeds in a component $K$  then $K$ embeds in $H_i$.  Suppose that $G' \subset G$ induces a sibling of $G$ via an embedding $\phi$.   Since $H_1$  is connected the image under $\phi$ must be a contained in some component $K$ of $G$, and by maximality $K$ is equimorphic to $H_1$.  But then $K$ and $H_1$ must be isomorphic.  Thus removing $H_1$ reduces the number of non-trivial components of $G$, and the result follows by induction.
%
% \begin{enumerate}
%
%\item  $G\oplus 1$ embeds in $G$ if and only if $H\oplus 1$ embeds in $H$ for some connected component. 
% 
% \item If $G\oplus 1$ does not embed  in $G$ and each connected component  has just one connected sibling then $G$  has one sibling. 
% \end{enumerate}
%
%\noindent{\bf Subsubcase 2.2.1} $H\oplus 1\not \leq H$ for every connected component $H$. Since also $\sib_{conn}(H)=1$ for every connected component, Lemma \ref{} ensures that  $\sib (G^c)=1$ It follows that $\sib (G)=1$. 
%
% \noindent{\bf Subsubcase 2.2.2} $H\oplus 1  \leq H$ for some  connected component $H$.
% 
%  
%embeds into some $H_i$. In this subcase, $G^c\bigoplus 1$ embeds into $G^c$. This implies that $G^c\bigoplus K^c_{n}$ embeds into $G^c$ for every integer $n$ (here $K_n$ is the complete graph on $n$ vertices). Hence,  all $G+ K_{n}$ embeds into $G$,  hence are non-isomorphic connected siblings of $G$.\end{proof} 
\begin{lemma}
\label{lem:more siblings}Let $G$ be a countable graph. Suppose that some connected component has infinitely many connected siblings. Then $G$ has infinitely many siblings and if it has at least two connected components, infinitely many are disconnected. 
\end{lemma}
\begin{proof}
Let $G$ satisfy the conditions of the lemma.  Let $H$ be a connected component with infinitely many non-isomorphic connected siblings $H_0=H, H_1, \dots, H_k, \dots$   Let $G_i$ result from $G$ by replacing every component that is a sibling of $H$ by $H_i$. Then $G_i$ is a sibling of $G$ and the $G_i$'s are pairwise non isomorphic. 

If G has has at least two components, the construction yields disconnected siblings.
\end{proof}%
%\begin{theorem}\label{label:theo1} For a graph $G$, $\sib(G)=1$ or infinite if $\sib_{conn}(C)=1$ or infinite for every connected component $C$ of  $G$.  
%\end {theorem}
%
%

\begin{lemma}
\label{lem:increasing chain}Let $G$ be a countable graph. Suppose that $G$ is disconnected and that there is an infinite  sequence $(H_n)_{n<\omega}$ of non-trivial components which is increasing w.r.t.  embeddability. Then $G$ has infinitely many disconnected siblings.
\end{lemma}

\begin{proof}
First, $G\oplus \overline{K_{\omega}}$ embeds into $G$. Indeed, if  $K$ is the union of the connected components distinct from the $H_n$'s, then $G= (\bigoplus_{n\in \N} H_n) \oplus K$. Since $\bigoplus_{n\in \N} H_n$ embeds into $\bigoplus_{n\in \N} H_{2n}$,  $G\oplus \overline{K_{\omega}}$ embeds into  $G$. Let $G'$ be the restriction of $G$ to the nontrivial components of $G$. Then $G'$  is equimorphic to $G$. And also $G'\oplus \overline K_{n}$. Thus, $G$ has infinitely many siblings and infinitely many are disconnected. 
\end{proof}

One can obtain a bit more under the stronger hypothesis that $H_n$ embeds in $H_m$ just in case $n<m$.  In this case, we can obtain a continuum of non-isomorphic siblings. Let $\mathfrak K$ be the family of components $K$ for which there is $n$ such that $K$ embeds in $H_n$.  Enumerate $\mathfrak K$ as $\{K_1,...,K_n,...\}$.  Let the union of the components not contained in $\mathfrak K$  be denoted $L$.  Let $J$ be an infinite subset of $\{1,2,\dots\}$. Set $G_J:= L \oplus \bigoplus_{n \in J}H_n$.
Since $K_m$ embeds in infinitely many $H_k$, and the same is true for $H_n$ for $n \notin J$, by interleaving we see that $G_J$ is  a sibling of $G$.

Now let $J_1 \neq J_2$, and let $i$ be the least element in the symmetric difference.  We may suppose that $i \in J_2 \setminus J_1$.  Then $G_{J_2}$ has a component isomorphic to $H_i$ while $G_{J_1}$ does not.  

In summary, with Lemma  \ref{lem:disconnected1}, \ref{lem:more siblings}  and \ref{lem:increasing chain}, we have:

\begin{proposition}\label{summary} A disconnected graph $G$ has infinitely many disconnected siblings provided that either $G$ is equimorphic to some connected component, or some connected component has infinitely many connected siblings, or there is an infinite  sequence $(H_n)_{n<\omega}$ of non-trivial components which is increasing w.r.t.  embeddability.  
\end{proposition}

\subsection{Tree decomposition of a cograph}
The crucial result we will use for the proof of Theorems \ref{thm:connected-siblings} and \ref {thm:cograph- one-sibling}  is the following. 
\begin{theorem}\label{label:chain/antichain} If a countably infinite  cograph is connected and its complement too then $G$ has $2^{\aleph_0}$  siblings with the same property. 
\end{theorem}
This relies on properties of the tree decomposition of a cograph.
Cographs have a simple structure. They can be obtained from the one vertex graph by iteration  of three operations: direct sum, complete sum and   sum over a labelled chain. 
If $(G_{i})_ {i\in I}$ is a family of at least two non-empty graphs, their \emph{direct sum} $\bigoplus_{i\in I} G_i$ is the disjoint union of the $G_i$'s with no edge between distinct $G_i$'s; their \emph{complete sum}  $\bigplus _{i\in I} G_i$ is the disjoint union of the $G_i$'s with  any pair of vertices between distinct $G_i's$ connected by an edge.  If there is a linear order $\leq$ on $I$ and a labelling $r$ of $I$ by $0$ and $1$, this structure being denoted $C:= (I, \leq, r)$, then the  \emph{labelled sum} denoted by $\sum_{i\in C} G_i$ is the disjoint  union of the $G_i$'s, two vertices $x\in G_i$, $y\in G_j$, with $i<j$, being linked by an edge if and only $r(i)=1$. We may view such a sum  as the graph associated to the labelled chain $C:= (I, \leq, \ell)$ where $\ell(i):= (G_i, r(i))$ and denote it by $\Sigma  C$. 
Due to the associativity of these operations, all possible sums do not need to be considered. The direct or complete sums we need to consider are the  direct sums of connected graphs, and the complete sums of graphs whose complements are connected. Since the complement of a finite connected cograph  is  disconnected, finite cographs are obtained by means of direct or complete sums. Infinite cographs may require labelled sums, but it suffices to consider those  indexed by infinite \emph{densely labelled} chains (that is chains in which the two labels occur in every interval having at least two elements) such that if the label of $i$ is $0$, resp. $1$,  then $G_{i}$ is not a complete sum, resp. a direct sum,  of at least two nonempty cographs, and if $i$ is the largest element of $C$, then $G_i$ is either a direct sum or a complete sum of at least two singletons. We will say that the labelled chain $C$ and  its  sum are  \emph{reduced}. 

With these operations, one has (see Boudabbous and Delhomm\'e \cite{boudabbous-delhomme}, Lemma 2.2 page 1747):
\begin{theorem}\label{decomposition-cographs}Let $G$ be a cograph with more than one vertex. Then either 
\begin{enumerate}
\item $G$  is direct sum of at least two  nonempty connected cographs, or
\item$G$ is a complete sum of at least two nonempty cographs whose complements are  connected, or
\item $G$ is the  sum $\Sigma C$ of a reduced labelled chain $C:=(I, \leq, \ell)$, where $(I, \leq)$ is  an  infinite chain with no first element, each label  $\ell(i)$ is the pair $(G_i, r(i))$ made of a non-empty  cograph $G_i$ and  an element $r(i)\in \{0,1\}$ in such a way that $r$ is a dense labelling of the chain $(I, \leq)$. \end{enumerate}
\end{theorem}
With some effort, one can (partially) evaluate the number of siblings of a direct sum or a complete sum in relation with the number of siblings of its components. For the case of labelled chains,  more substantial  information is needed. This information comes from the tree decomposition  of a graph, that we  consider here only for cographs. 
The \emph{tree decomposition} of a cograph $G$ is a labelled tree ${\bf T}(G):= (R(G), v)$ defined as follows:
The nodes of the tree $R(G)$ are the  robust modules of $G$; a  module $A$ is \emph{robust} if it  is the least strong module containing two vertices $a$, $b$ of $G$ (a module is \emph{strong} if it is either comparable to or disjoint from  every module).  Non-trivial robust modules are labelled with one of two symbols $0$ and $1$. If $A$ is the least robust module containing two distinct vertices $a$ and  $b$  of  $G$, the label $v(A)$ is  $1$ if $\{a,b\}$ forms an edge, whereas it is $0$ if $\{a,b\}$ does not form an  edge (it turns out that the  label does not depend upon the choice of $a$ and $b$). The order on the nodes of the tree is reverse inclusion. The graph $G$ can be recovered from the labelled tree. A description of these labelled trees is given in the Appendix. An example made of a clique and an independent set (with some extra edges) is given in Figure  \ref{Examples2}. 

\begin{center}
\begin{figure}[ht]
\includegraphics[width=4in]{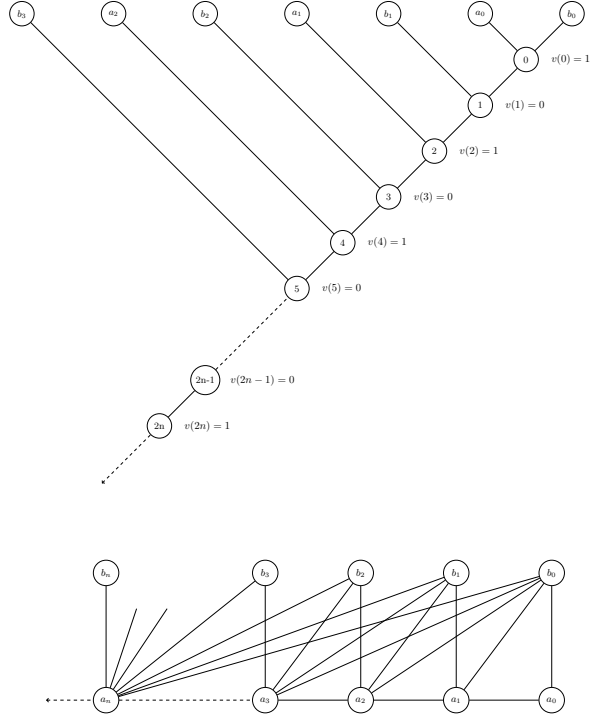}
\caption{A labelled tree with no root and the corresponding cograph}
 \end{figure}
 \label{Examples2}
\end{center}

The first two cases of Theorem \ref{decomposition-cographs} correspond to the case of a cograph whose  decomposition tree has a least element; the last case to the non-existence of a least element.

The stated condition in Theorem \ref{label:chain/antichain} amounts to the fact that the decomposition tree has no least element. The conclusion follows from the next two lemmas.

\begin{lemma}\label{lem:initialsegment}
Let $G$, $G'$ be two isomorphic cographs such that their tree decomposition has no least element. If $G= \Sigma C$  and $G'= \Sigma C'$ where  $C:=(I, \leq, \ell)$ and $C':=(I', \leq', \ell')$ are  two reduced labelled chains then there are two infinite  initial segments $W$ of $I$ and $W'$ of $I'$  and an isomorphism $h$ of the induced labelled chains $C_{\restriction W}$  and  $C'_{\restriction W'}$. \end{lemma}

% G'_{h(i)}$ and $G_{i}$ are isomorphic and the labels $v'(h(i))$ and  $v(i)$ are equal for every $i\in W$.
\begin{proof}Since $G$ and $G'$ are isomorphic, their decomposition trees are isomorphic. We may suppose that they are identical. Let $T$ be such a tree. Then $(I, \leq)$  and $(I', \leq )$ correspond to  maximal chains  in $T$. If these chains are identical, there is nothing to prove ($W=W'=I$). If not, then  since $T$ has no least element  these chains meet in some element, say $a$. Set $W= W':= \downarrow a$ and $h$ be the identity on the initial segment $\downarrow a$.
\end{proof}

\begin{lemma}\label{lem:main} Let $C:= (I,  \leq, \ell)$ be a countable reduced labelled chain with no least element. 
Then there are $2^{\aleph_0}$ reduced  labelled chains with no least element $C_{\alpha}:=  (I_{\alpha}, \leq_{\alpha}, \ell_{\alpha})$  such that: 
\begin{enumerate}
\item all their labelled sums are siblings of the labelled sum of $C$; 
\item there are no distinct $\alpha$, $\alpha'$ and no nonempty  initial segment  $W$ of $I_{\alpha}$ and $W'$ of $I_{\alpha'}$ such that the induced labelled chains on $W$ and $W'$ are isomorphic. 
\end{enumerate}
\end{lemma} 

\subsection{Sketch of  the proof of Lemma \ref{lem:main}}\label{subsection:sketch} 

We start with a countable reduced labelled chain $C:= (I, \leq, \ell)$ with $\ell (i):= (G_i, v(i))$ for each $i\in I$. We suppose that $(I, \leq)$ has no least element. 

We select an initial segment $J$ of $I$ such that the restriction $C_J:= (J, \leq_{\restriction J},  \ell_{\restriction J})$ is indecomposable on the left (see Subsection \ref{subsectionlabelled} below for the definition and existence). Let $Ev(J)$ be the set of $i\in J$ such that  $G_i$ is  a clique or an independent set of even size.  To each map $f:\N \rightarrow \{0, 1\}$ we associate a reduced labelled chain $C_f$ as follows.  
We select an infinite descending sequence $(a_n)_{n\in \N}$ of elements of $J$  coinitial in $J$ and such that $v(a_n)= 1$ for each $n\in \N$. Now we insert   infinitely many elements $b_n, c_n$ such that $b_n$ covers $a_n$ and $c_n$ covers $b_n$.   Let $\overline J:= J\cup \{b_n, c_n: n\in \N\}$, $K:= I\setminus J$ and  $\overline I= \overline J \cup K$.  
Let  $f:\N \rightarrow \{0,1\}$. We define a labelling $\ell_f$ on $\overline I$, with $\ell_f(i):= (\overline G_i, v_f(i))$  as follows. If  $i\in I\setminus Ev(J)$ we set $\ell_{f}(i)= \ell (i)$. If  $i\in Ev(J)$, then $v_f(i)= v(i)$ and  $\overline G(i)$ extends $G_i$ to an extra element in such a way  that the new graph is a clique if $G_i$ is a clique and an independent set, otherwise.  If $i\in \{b_n, c_n\}$, recalling that $v(a_n)= 1$,  we set  $v_f (b_n)=0, v_f(c_n)= 1$, 
 $\overline G_{i}$ being an independent set, resp.,    a clique, if $i= b_n$, resp. $i= c_n$,  of  size $2f(n)+2$. 
 Due to the labelling,  the labelled chain $C_{f}:= (\overline I, \leq, \ell_{f})$ is reduced. If $f$ and $f'$ are two maps from $\N$ to $\N$ such that the sums $\Sigma C_f$ of $C_f:= (\overline I, \leq, \ell_f)$ and $\Sigma C_{f'}$ of $C_{f'}:= (\overline I, \leq, \ell_{f'})$  are isomorphic then according to Lemma \ref{lem:initialsegment} there are two initial  segment $W$ and $W'$ of $\overline I$ and an  isomorphism $h$ from $W$ onto $W'$ preserving the labels, that is the labels $\ell_f (i):=(\overline G_i, v(i))$ and $\ell_{f'}(h(i)):=(\overline G_{h(i)}, v(h(i))$ are isomorphic, meaning that $\overline G_i$ and $\overline G_{h(i)}$ are isomorphic and $v_{f}(i)= v'_{f}(h(i))$ for each $i\in W$.  If $i= b_n$ then necessarily $h(b_n)=b_m$ for some $m$. Indeed, $\overline G_{h(i)}$ must be an independent set of even size, thus $h(b_n) \not \in Ev(J)$ hence $h(b_n)$ must be either some $b_m$ or some $c_m$. Since $v(h(b_n))=v(b_n)=  0$ this is some $b_m$. Consequently,  there are two final segments  $D$ and $D'$ of $\N$ and an isomorphism $h$ from $D$ to $D'$ such that $f'(h (n))=f(n)$. Thus, in order to obtain  $2^{\aleph_0}$ non isomorphic reduced labelled chains, we may apply the following lemma. 
 
 \begin{lemma}\label{laflamme-al} There is a set of  $2^{\aleph_0}$ maps $f$ from $\N$ into $\{0,1\}^{\N}$ such that  for every pair $f$,  $f'$  of distinct maps, and every isomorphism $h$ from a final segment  $D$ onto another  $D'$ there is some $n$ such that $f(n)\not = f'(h(n))$. 
 \end{lemma}
 
 This lemma appears as  Lemma 4 p. 39 of \cite{laflamme-al}. For the reader's convenience we reproduce the short proof given there. 
 
\begin{proof}We start with a subset $X:=\{x_n: n \in \N\}$ where $x_0=0$  and $x_{n+1}=x_n + n$ (one just need the gaps increasing). Now, let $\mathcal A$  be an almost disjoint family of $2^{\aleph_0}$ infinite subsets of  $X$.  For any $A\in \mathcal A$, and $n>0$, $A+n$ is almost disjoint from $X$, and thus almost disjoint from any other $A'$. The characteristic functions of members of $\mathcal A$ have the required property. 
\end{proof}

In order to complete the proof of Lemma \ref{lem:main} it suffices to prove  that each sum $\Sigma C_{f}$ of $C_{f}$ is  equimorphic to the sum $\Sigma C$ of $C$.  This is quite easy if in $C$ no  $G_i$ has an even size (e.g. each $G_i$ is finite but of odd size or is infinite). Some difficulties arise in general. Our  proof,  given in Lemma \ref{lemma:equimorphy} of the next Subsection,  is based on properties of chains labelled by a  quasi order $Q$.  
 
 \subsection{Chains labelled by a better quasi order}\label{subsectionlabelled}
We consider below chains labelled by a quasi-ordered-set.  There is a strong similarity between properties of countable chains and properties of labelled countable chains, provided that the quasi-ordered-set  is a better quasi ordering (in short b.q.o.).  We present the facts we need. We refer to chapter 10 of \cite{rosenstein} and to chapter 6 and 7 of \cite {fraisse} for properties of chains and an exposition of the solution of Fra\"{\i}ss\'e's conjectures by Laver. We refer to the papers of Laver \cite{laver, laver2} and also to \cite{pouzet70} for properties of labelled chains. 

A \emph{labelled chain} is a pair   $C:= (I, \ell)$ where  $\ell$ is a map from $I$ to a quasi-ordered-set $Q$.  A
$Q$-\emph{embedding} of a labelled chain $C:=(I,\ell)$ into another 
$C':=(I',\ell')$ is an order embedding   $f:I \rightarrow I'$ such that $\ell(x)
\leq \ell'(f(x))$ for all $x \in I$, a fact that we denote by $C\leq C'$. These 
two labelled chains  are \emph{equimorphic}, and we set $C\equiv C'$,  if they are $Q$-embeddable in each other. They are 
\emph{isomorphic} if there is an isomorphism $\phi$ from $I$ to $I'$ such
that $\ell' \circ \phi = \ell$.
A labelled chain $C:= (I, \ell)$   is \emph{additively indecomposable}, or briefly \emph{indecomposable},  if for every partition of $I$ into an initial interval $I$ and a final interval   $F$,  $C$ embeds into $C_{\restriction I}$ or into $C_{\restriction F}$. We say that $C$ is \emph{left-indecomposable}, resp., \emph{strictly left-indecomposable},  if $C$ embeds in every non-empty initial segment, resp. if $C$ is left-indecomposable  and embeds in no proper final segment. 
The {\it right-indecomposability}
\index{right-indecomposable} and the {\it strict
right-indecomposability} \index{strictly right-indecomposable} are
defined in the same way. The notions of indecomposability and
strict-right or left-indecomposability are preserved under
equimorphy (but not the right or the left-indecomposability). 

The \emph{sum} $\Sigma_{i\in I} C_i$ of labelled chains over a chain (not over a labelled chain) is defined as in the case of chains. If $I$ is the $n$-element chain $\underline n:= \{0, 1, \dots,  n-1\}$ with $0<1< \dots <n-1$, the sum is rather denoted by $C_0+C_1+ \cdots + C_{n-1}$.  With the notion of sum, a labelled chain $C$ is indecomposable if $C\leq C_0+C_1$ implies $C\leq C_0$ or $C\leq C_1$.
% If a labelled chain  $C$ can be written
%under the form $C=C_{0}+C_{1}+\ldots+C_{n-1}$
%with each $C_{i}$ indecomposable, then it is a {\it
%decomposition} \index{decomposition} of $C$ of {\it length}
%\index{lenght} $n$. This decomposition is {\it canonical} \index{
%canonical} if its length is minimal. 
A sequence $(C_n)_{n<\omega}$ of labelled chains  is {\it
quasi-monotonic} \index{quasi-monotonic} if $\{m: C_{n} \leq
C_{m} \}$ is infinite for each $n$, its sum
$C:=\Sigma_{n<\omega} C_{n}$ is right-indecomposable or
equivalently $C^*=\Sigma^*_{n<\omega}C_{n}$ is
left-indecomposable.

  Let $C:= (I, \leq, \ell)$ be a labelled chain. Two elements $x, y$ of $I$ are \emph{equivalent} and we set $x\equiv_C y$ if $C$ does not embed in  the restriction of $C$ to the interval determined by $x$ and $y$. If $C$ is indecomposable  then this relation is an equivalence relation. Equivalence classes are intervals of $I$. Moreover, either all the elements of $I$ are equivalent, that is there is just one equivalence class, or the quotient $I/\equiv_C$ is dense and if $F$ is any equivalence class,  $C$ does not embed into $C_{\restriction F}$.

  The following lemma lists the properties we need. The first three are due to Laver, see  \cite{laver} p. 109 for the first two and \cite{laver2} p. 179  for the third. For the reader's convenience, we give a (short) proof.  For this purpose, let  $\mathcal D_{\leq \omega}$ be the collection of countable chains quasi-ordered by embeddability and let $Q^{\mathcal D_{\leq \omega}}$ be the collection of countable chains, labelled by a quasi-ordered set $Q$.  According to  Laver \cite{laver}, $\mathcal {D}_{\leq \omega}$  quasi-ordered by embeddability is better-quasi-ordered (b.q.o.) and  more strongly,  the collection $\mathcal Q^{\mathcal {D}_{\leq \omega}}$ of  countable chains labelled by a b.q.o $Q$ is b.q.o.(see p. 90 of  \cite{laver}).

 \begin{lemma}\label{lemma:basicindec}Let $C:= (I, \ell)$ be a countable chain labelled over a b.q.o. $Q$. Then 
 \begin{enumerate}
 \item $C$ is   a finite sum of indecomposable labelled chains; 
  
  \item If  $I$ has no least element, then there is  some initial interval $J$  such  that $C_{\restriction J}$ is left-indecomposable; 
  
 \item  If $C$ is left-indecomposable then $C$ is an $\omega^*$ sum $\Sigma^{*}_{n<\omega} C_n$ where each $C_n$ is indecomposable and the set of $m$ such that $C_n$ embeds into $C_m$ is infinite;  
 
 \item If $C$  is indecomposable and  the quotient $I/\equiv_C$ is dense then $C$ is equimorphic to a sum $\Sigma_{q\in \Q} C_q$ such that for every $p<q$ and $r$ in $\Q$ there are  $s_0, \dots, s_{n_{m-1}}$ with $p<s_0<\cdots < s_{n_{m-1}}<q$ and $C_r\leq C_{s_0} + \cdots +C_{s_{n_{m-1}}}<C$.  

\end{enumerate}
 \end{lemma}
 \begin{proof} 
 $(1)$ Since $\mathcal Q^{D_{\leq \omega}}$ is b.q.o. it is w.q.o. hence every  non-empty  subset has a minimal element. If $(i)$ fails,  it fails for some minimal $C$. We claim that $C$ is indecomposable. Indeed, if $C= C_0+C_1$ with $C_0<C$ and $C_1<C$ then $C_0$ and $C_1$ are finite sums of indecomposables thus $C$ is such, contradicting our hypothesis. Hence $C$ is indecomposable, contradicting our hypothesis too. Hence $(2)$ holds.   
 
 $(2)$  For each $a\in I$, let $C_a:= (I_a, \ell_{\restriction I_a})$  where $I_a:= (\leftarrow a]$ ($:= \{x\in I: x\leq a\}$, a set that  we denote also $\downarrow a$). As a subset of a w.q.o. the set  of these $C_a$ has a minimal element. This minimal element is left-indecomposable. 

$(3)$ Pick a coinitial sequence in $I$, say $(x_n)_{n<\omega}$. Write $I$ as $I:= \bigcup_{n<\omega} I_n$ where $I_{0}:= [x_0, \rightarrow[$, $I_{n+1}=[x_{n+1}, x_{n}[$. Write $C$ as   $\Sigma^{*}_{n\in \omega} X_n$, where $X_{n}:=( I_n, \ell_{\restriction I_n})$. Since, by $(i)$,  each $X_n$ is a finite sum of indecomposables, we may rewrite $C$ as a sum  $\Sigma^{*}_{n<\omega} C_n$ where each $C_n$ is indecomposable. Pick $n<\omega$ and any $k>n$. Since $C$ is left-indecomposable, $C$ embeds into  
$\Sigma^{*}_{k< m<\omega} C_m$, hence $C_{n+1}+C_{n}$ embeds in that sum, thus $C_{n}$ embeds in a finite sum of $C_m$, $m>k$. This $C_{n}$  being indecomposable, it embeds into some $C_m$.

$(4)$  Let $D:= (I/\equiv_C, \ell_{\equiv_C})$ be the labelled chain where $\ell_{\equiv_C}(F):= (F, \ell_{\restriction F})$ where $F\in D$.  The set of labels belongs to $\mathcal Q^{D_{\leq \omega}}$ hence is b.q.o. As a labelled chain over a b.q.o.   $D$ is a finite sum of indecomposable labelled chains. The decomposition of $I/\equiv_C$ in finitely many intervals on which the labelled chains are indecomposable induces a decomposition of $I$ into finitely many intervals. Since $C$ is indecomposable, $C$ is equimorphic to its restriction to some interval.  Since $C_{\restriction F}< C$ for each $F\in I/\equiv_C$ we may throw out the  extremities of this interval if any, hence we may suppose that this interval is isomorphic to $\Q$. Since each $C_{\restriction F}$ is a finite sum of  indecomposable labelled chains and $C$ embeds in $C_{\restriction [x, y]}$ for every $x<y$ with $x\not \equiv_C y$,  the conclusion follows.  

 \end{proof}

 Let $C:= (I, \ell)$ be a  chain labelled by a poset $Q$ and $n$ be a positive integer, the \emph{ordinal product} $\underline n.C$ is the labelled chain $(\underline n. I, \ell_n)$  in which $\underline n. I$ is the ordinal product of the $n$-element chain $\underline n:= \{0, \dots n-1\} $ and the chain $I$, that is the ordinal sum of $C$ copies of $\underline n$, and $\ell_{n} (m, i)= \ell(i)$ for every $m\in \{0, n-1\}$, $i\in I$.  
 
 \begin{lemma}\label{lem:indec} Let $C:= (I, \ell)$ be a countably infinite  labelled chain. If the labels belong to a b.q.o.  and  
 $C$ is indecomposable then for every positive integer $n$, the ordinal product $\underline n. C$ embeds in $C$. 
 \end{lemma} 
 \begin{proof} It suffices to prove  that the property holds for $n=2$. Indeed, suppose that the property  holds for $n$. Let $C$ be an indecomposable chain, then trivially $(\underline {n+1}).C$ embeds into $2n. C$. Since $\underline {2n}.C= \underline {n}.(\underline {2}. C)$ and,  as it can be checked easily, $\underline {2}.C$ is indecomposable, induction ensures that $\underline {n}.(\underline {2}. C)$ embeds into $\underline {2}.C$. With the fact that  $\underline {2}.C$ embeds into $C$ we obtain that $(\underline {n+1}).C$  embeds into $C$. 
 
 To prove that this property holds,  we use induction. Since the collection of labelled countable chains over a b.q.o. is b.q.o. it is well founded, hence we may suppose that $C$ is a chain such that  every countably infinite indecomposable labelled chain $D$ satisfying $D<C$ satisfies the property. 
 
% If there is no such $D$ then $I$ has order type $\omega$ or $\omega^*$ ( apply Item $(ii)$ of Lemma  \ref{lemma:basicindec}). Suppose that $I$ has order type $\omega^*$, an embedding or $2.C$ is readily obtained by induction . The case $\omega$ is similar. Thus the property holds in these cases. 
% 
We consider two cases. 
 
\noindent {\bf Case 1.} The equivalence relation $\equiv_C$ has just one class. Then $C$ is either strictly left-indecomposable or strictly right-indecomposable. 
 We may suppose that $C$ is strictly left-indecomposable. We apply Item $(iii)$ of Lemma  \ref{lemma:basicindec}:  $C$ is equimorphic in an $\omega^*$ sum $\sum^{*}_{n<\omega} C_n$ where each $C_n$ is indecomposable and the set of $m$ such that $C_n$ embeds into $C_m$ is infinite. Hence $\underline 2. C=\sum^{*}_{n<\omega}\underline 2.  C_n$. We define an embedding of $\underline 2.C$ into $C$ as follows. Suppose that we have embedded $\sum^{*}_{n< m} \underline 2.C_n$ in $\sum^{*}_{n<\varphi (m)} C_n$. We extends this embedding to $\underline 2.C_m$ as follows. If $C_{m}$ is infinite,  we embeds it into some $C_{k}$ with $k\geq \varphi (m)$. According to the induction hypothesis, $\underline 2.C_{k}$ embeds into $C_{k}$, and we may set $\varphi(m+1)= k$. If $C_m$ is a one element labelled chain, $\underline 2.C_m= C_m+C_m$,  we send  the first copy of $C_m$ into some $C_k$ and the second copy into another $C_k'$ for $k'>k\geq \varphi(m)$.  
 
\noindent {\bf Case 2.} The equivalence  relation has at least two classes, and in fact a  dense set of classes. According to  $(iv)$ of Lemma \ref{lemma:basicindec}, $C$ is equimorphic to a sum $\Sigma_{q\in \Q} C_q$ such that for every $p<q$ and $r$ in $\Q$ there are  $s_0, \dots, s_{n_r-1}$ with $p<s_0<\cdots < s_{n_r-1}<q$ and $C_r\leq C_{s_0} + \cdots +C_{s_{n_r-1}}<C$. Let $p_0, \dots, p_n, \dots $ be an  enumeration of $\Q$. Suppose that we have defined an embedding $\varphi_m$ of $\Sigma^{*}_{p_n <p_m} \underline 2.C_{p_n}$ in a  sum $\Sigma_{q\in A} C_q$ where $A$ is a finite subset of $\Q$ and in such a way that the projections on $\Q$ of the images of the $\underline 2.C_{p_n}$'s do not intersect. We extend it to  $\underline 2.C_{p_m}$ as follows. Let $A^-$ be the initial,  resp. $A^+$ be the final, segment  of $\Q$ generated by the projections on $\Q$ of the images via $\varphi_m$ of the $\underline 2.C_{p_n}$'s for $p_n<p_m$, resp., for $p_m< p_n$. The complement is a dense interval of $Q$.  The labelled chain $\underline 2.C_{p_m}$  is a finite sum $\underline 2.C_{s_0} + \cdots + \underline 2. C_{s_{m-1}}$, the $C_{s_i}$ being indecomposable. If  $C_{s_i}$ is infinite , we may send  $\underline 2.C_{s_i}$ into some $C_j$ and if $C_i$ is a one element labelled chain, we may send $\underline 2.C_{s_i}$ into a sum $C_i +C_j$. 
 \end{proof}

Let $\Cog_{\leq \omega}$  be the collection of countable cographs, quasi-ordered by embeddability. According to Thomass\'e \cite{thomasse}, $\Cog_{\leq \omega}$ is b.q.o. Let  $Q:= \Cog_{\leq \omega}\times \{0,1\}$ be  the direct product of $\Cog_{\leq \omega}$ and the two element antichain $\{0,1\}$. This poset is b.q.o. as a union  of two b.q.o.'s.   And  thus,  from Laver's theorem, $Q^{\mathcal D_{\leq \omega}}$ is b.q.o.

 \begin{lemma} \label{lemma:equimorphy}Let  $C:= (I, \leq, \ell)$ be a countable reduced labelled chain such that $I$ has no first element  and the labels  belong to $\Cog_{\leq \omega}\times \{0,1\}$.  Then,  there is an initial segment $J$ of $I$ such that the restriction $C_J:= (J, \leq_{\restriction J},  \ell_{\restriction J})$ is left-indecomposable. If $J$ is such an initial segment then,  for every  map $f:\N \rightarrow \{0, 1\}$,   the sum $\Sigma  C_f$ of $C_f$ is embeddable into   the sum $\Sigma  C$ of $C$. 
  \end{lemma}
  
  \begin{proof}
  The existence of $J$ follows from the fact that the set of countable chains labelled by   $\Cog_{\leq \omega}\times \{0,1\}$ is b.q.o. and $(ii)$ of Lemma \ref{lemma:basicindec}. 
  Let  $f:\N \rightarrow \{0,1\}$ and  $C_{f}:= (\overline I, \leq, \ell_{f})$ be the labelled chain defined in Subsection \ref{subsection:sketch}. We prove that $\Sigma  C_f\leq \Sigma  C$. 
Let ${C_{f}}_{\restriction {\overline J}}$, resp.,  ${C_{f}}_{\restriction K}$ the restriction of $C_{f}$ to $K$. We have  $C_f= {C_{f}}_{\restriction \overline J} +   {C_{f}}_{\restriction K}$. As it is easy to see,  $\Sigma C_f$ extends both  $\Sigma {C_{f}}_{\restriction \overline J}$ and $\Sigma {{C_{f}}_{\restriction K}}$ in a simple way: If $a\in  \Sigma {{C_{f}}_{\restriction \overline J}}$  and $b\in \Sigma  {{C_{f}}_{\restriction K}}$, we link $a$ and $b$ by an edge if $v_f(i)=1$ where $i \in {\overline J}$ and  $a\in \overline G_i$. We denote by $\overline +$ this operation, hence   $\Sigma C_f= \Sigma {C_{f}}_{\restriction \overline J} \overline{+}  \Sigma  {C_{f}}_{\restriction K}$.   
  
To conclude, it suffices to prove that $\Sigma C_{f \restriction \overline {J}}$ embeds into $\Sigma C_{\restriction J}$ and $\Sigma {C_{f}}_{\restriction K}$ embeds into $\Sigma {C_{f}}_{\restriction K}$. Only the first statement needs a proof.
In order to prove it, we define an auxiliary labelled chain $D:= (L, \leq , d)$  as follows. The domain $L$ is $\underline {2}.J \cup X$ where $X:= \{b_n,c_n: n<\omega\}$,  $b_n$ covers $(1, a_n)$,  $c_n$ covers $b_n$ and $(a_n)_n$ is the sequence coinitial in $J$ defined in Subsection \ref{subsection:sketch}. The labelling $d$ is defined by $d(i):= \ell (r)$ if $i:= (j, r)\in \{0,1\}.J$, $d(i):= \ell_f(i)$ if $i\in \{a_n, b_n\}$ (So $D$ is not densely labelled, but this does not matter).  
We prove that the following inequalities hold. 
  
$$\Sigma {C_{f}}_{\restriction \overline {J}}\leq \Sigma D \leq \Sigma \underline 2. C_{\restriction J}\leq \Sigma C_{\restriction J}.$$

For the first inequality, note that by definition $\Sigma {C_{f}}_{\restriction \overline {J}}= \Sigma_{i\in (\overline J, v_f) }\overline G_i$ and $\Sigma D= \Sigma_{j\in L} D_j$ such that $D_j$ is the first component of $d(j)$. We have $\Sigma D= \Sigma_{i\in (\overline J, v_f)} H_i$ with $H_i= \overline G_i$ if $i\in \{a_n, b_n\}$ and $H_i= G_{0, i} \oplus G_{1, i}$,  resp. $G_{0, i} + G_{1, i}$ if $v_f(i)=0$, resp. $v_f(i)=1$ if $i\in I$. It follows that $\overline G_i \leq H_i$ for every $i\in \overline J$. Thus $\Sigma {C_{f}}_{\restriction \overline {J}}\leq \Sigma D$. 

For the  second inequality,  we observe that ${C}_{\restriction J}$ being left-indecomposable, we  may write it as an $\omega^*$ sum $\Sigma^{*}_{n<\omega} C_n$ where each $C_n$ is indecomposable and the set of $m$ such that $C_n$ embeds into $C_m$ is infinite. For each $n$, let $J_n$ be the domain of $C_n$ and let $L_n:= \underline {2}.J_n \cup \{b_m, c_m: (1, a_m)\in   \underline {2}.J_n\}$. Set $D_n:= D_{\restriction L_n}$. Then  $D=\Sigma^{*}_{n<\omega} D_n$. 
Define an  embedding from $\Sigma D$ in $\Sigma \underline 2. C_{\restriction J}$ by induction. Let $n<\omega$. Suppose that $\Sigma(D_{n-1} +\cdots + D_{0})$ has been embedded  in $\Sigma(\underline 2. C_{\varphi (n  -1)} +\cdots + \underline 2. C_{0})$.  The labelled chain  $D_{n}$ consists of $\underline 2. C_{n}$ plus finitely many elements, each labelled by a finite graph. Thus $D_n$ can be written  
$D_{n,0} + \alpha_{(0,0)}+ \alpha_{(0,1)}  +D_{n,1} + \alpha_{(1,0)} + \alpha_{(1,1)}+ \cdots + D_{k-2}+\alpha_{(k-2,0)} + \alpha_{(k-2,1)}+D_{k-1}$ with $\underline{2}. C_{n}= D_{n,0}+  \cdots + D_{n, k-1}$ and the $\alpha_i$ are singletons belonging to $X$ labelled by a $2$ or $4$-element cograph. Since the set of $m$ such that $C_n$ embeds in $C_m$ is infinite, of the $C_m$ we may find $m\geq \varphi(n-1)$ and $k$ such that  $\Sigma D_n$ embeds in  $\Sigma (D_{m+k} +\cdots +  D_{m})$. Then set $\varphi(n) = m+k$.

The last inequality is because $\underline 2. C_{\restriction J}\leq C_{\restriction J}$, a fact which follows  from Lemma \ref{lem:indec} since $C_{\restriction J}$ is indecomposable.

  With  that the proof of the lemma is complete.

  \end{proof}
  
  %We relate the number of siblings of a direct sum  of connected graphs to the number of siblings of its components. For an example, we prove that:
%
%\begin{theorem}\label{thm:thomasse} Let $G$ be a countable cograph. Then $\sib(G)\in \{1, \aleph_0, 2^{\aleph_0}\}$ provided that $\sib_{conn}(C)\in \{1, \aleph_0, 2^{\aleph_0}\}$ for every connected component $C$ of $G$.
%\end{theorem}
%
%\begin{problem}Extend this result to any countable graph. 
%\end{problem}
%
%
%\begin{theorem}\label{label:theo1} For a graph $G$, $\sib(G)=1$ or infinite if $\sib_{conn}(C)=1$ or infinite for every connected component $C$ of  $G$.  
%\end {theorem}
%We deduce  this fact  from Theorem \ref{main1theo} below.  
%We do not know if we can replace the hypothesis on $\sib_{conn}(C)$ by the same on $\sib(C)$ (see Remark \ref{remark1} below). 
%

\section{Proof of  Theorem  \ref{thm:connected-siblings}.}
We use induction on the collection $\Cog_{\leq \omega} $ of countable cographs with a quasi-order slightly  different from embeddability.  If countable  $G$ and $G'$  are two  connected cographs, we  set $G\preceq G'$ if $G\leq G'$ or $G\leq G'^c$.  This is a quasi ordering; since it extends the embeddability, which is w.q.o., it is w.q.o. hence well-founded. So, in order to prove that  a connected countable cograph $G$ has one or infinitely many connected siblings, we may suppose that this property holds for  all  connected countable cographs  $G'$  such that $G' \prec G$, that is either  $G'$ strictly embeds into $G$ or into $G^c$.   There are two cases to consider:

\noindent {\bf Case 1.} The complement $G^c$ of $G$ is not connected.
We  apply Proposition \ref{summary}. 
We decompose $G^c$ into connected components, say $G^c := \bigoplus_i H_i$,  where each $H_i$ is connected.    

We have  $H_i\leq G^c$, hence $H_i\preceq G$. 

\noindent {\bf Subcase 1} $G\preceq H_i$ for some $i$.  This means $G\leq H_i$ or $G^c\leq H_i$. This implies that $G^c$ is equimorphic to some  connected graph $K$ ($K= H_i$ in the second case, $K=G$ in the first  case). It follows from Lemma \ref{lem:disconnected1} that $G^c$ has infinitely many disconnected siblings, hence $G$ has infinitely many connected siblings. 
\noindent {\bf Subcase 2} $H_i \prec  G$ for  all $i$.

According to the induction hypothesis, each $H_i$ has either one or infinitely many connected siblings.  

\noindent{\bf Subcase 2.1} Some connected component of $G^{c}$ has infinitely many connected siblings. According to Lemma \ref{lem:more siblings}, $G^c$ has infinitely many disconnected siblings. It folows that $G$ has infinitely many connected siblings. 

\noindent {\bf Subcase 2.2 } Every connected component $H_i$ of $G^c$ has only one connected sibling.

\noindent{\bf Subcase 2.2.1} The number of non-trivial 	connected components of $G^c$ is infinite. Since the collection of countable cographs is w.q.o. under embeddability, there is an increasing sequence  
among these connected components. According to Lemma \ref{lem:increasing chain} $G^c$ has infinitely many disconnected siblings, hence $G$ has infinitely many connected siblings.

\noindent {\bf Subcase 2.2.2}. $G^c$ has only finitely many non-trivial connected components. We apply Lemma \ref{lem:onesiblingdirectsum}. Then $G^c$ has either one sibling or infinitely many disconnected siblings. Hence $G$ has either one sibling or infinitely many connected siblings.

\noindent {\bf Case 2. } 
The complement $G^c$ of $G$ is connected. In this case, the decomposition tree $T(G)$ has no least element. 

If $T(G)$ has no least element, then $G$ is a sum  $\sum_{i \in (C, \ell )} G_i$ of non-empty graphs, where $C$ is a chain with no first element, $v $ is a dense labelling of $C$ by $0$ and $1$, $G$ extends each $G_i$, and for $x\in G_i$, $y\in G_j$ with $i<j$, $\{x,y\}$ forms an edge  iff $v(i)=1$. According to Theorem \ref{label:chain/antichain}, $G$ has $2^{\aleph_0}$ connected siblings. 
\hfill $\Box$
\section{Proof of Theorem \ref{thm:cograph- one-sibling}} 

We prove $(i)\Rightarrow (ii)\Rightarrow (iii)\Rightarrow (i)$.

The implication $(i) \Rightarrow (ii)$ is trivial. Implication $(iii)\Rightarrow (i)$ is an immediate consequence of the following result. 
\begin{theorem}\label{thm:lex-one-sibling} If a graph $G$ is a finite lexicographic sum of cliques or independent sets then it has just one sibling: itself. 
\end{theorem}

This is a consequence of properties of monomorphic decompositions of relational structures, a notion introduced in \cite{pouzet-thiery} and developped in \cite{oudrar-pouzet, sikaddour-pouzet, oudrar, Lafl-Pouz-Saue-Wood}.

Let $R:= (V, (\rho_{i})_{i\in I})$ be  a relational structure. A \emph{monomorphic decomposition} of $R$  is any partition $(V_j)_{j \in J}$ of $V$ such that for every pair of finite subsets $F,F'$ of $V$, the restrictions $R_{\restriction F}$ and $R_{\restriction F'}$ are isomorphic whenever $\vert F\cap V_j\vert=\vert F'\cap V_j\vert$ for every $j\in J$. 
Among the monomorphic decompositions of $R$ there is a largest one: every other is included in it (see \cite{pouzet-thiery}, Proposition 2.12).  We call it the \emph{canonical decomposition} of $R$ and denote it by $\Mon(R)$. Its existence  is a consequence of the following notion: say that two elements $x,y$ of $V$ are \emph{equivalent} and set $x\simeq_R y$ if for every finite subset $F$ of $V\setminus \{x,y\}$ the restrictions $R_{\restriction F\cup \{x\}}$ and $R_{\restriction F\cup \{y\}}$ are isomorphic. 
%
%If this condition holds  for subsets of size at most $k$, we say that $x$ and $y$ are $(\leq k)$-equivalent. It was proved by Oudrar and Pouzet \cite {oudrar}, and independently   Boudabbous \cite{boudabbous}, that, for binary structures, $(\leq 6)$-equivalence coincide with equivalence. 

Oudrar and Pouzet showed (see Lemma 7.48 and Lemma 7.49 in Section  7.2.5 of \cite{oudrar}):

\begin{theorem}\label{thm:canonical} The partition of the domain $V$ of a relational structure $R$ into equivalence classes forms a monomorphic decomposition of $R$ and every other monomorphic decomposition of $R$ is a refinement of it. 
\end{theorem} 
We will need the following fact.

\begin{proposition}\label{prop:mono-dec}Let $R:= (V, (\rho_{i})_{i\in I})$ be a relational structure and $A$ a subset of $V$. Then 
\begin{enumerate}
\item Every monomorphic decomposition of $R$ induces a monomorphic decomposition of $R_{\restriction A}$; 

\item If $R$ has a monomorphic decomposition into finitely many classes and $R$ embeds into $R_{\restriction A}$ then for each class $C$ of $\Mon (R)$, $R_{\restriction A\cap C}$ is a class of the canonical decomposition of $\Mon (R_{\restriction A})$ and $\vert A\cap C\vert = \vert C\vert $. 
\end{enumerate}

\end{proposition}
\begin{proof} Item $(1)$ is obvious.

\noindent $(2)$  Let $f$ be an embedding of $R$ into $R_{\restriction A}$,  $A'$ be the range of $f$ and $R':= R_{\restriction A'}$.  According to $(1)$, $(A'\cap C)_{C\in \Mon (R)}$ is  a monomorphic decomposition of $R'$, hence it is finer than $\Mon(R')$.  Thus,  it has as many classes as $\Mon(R')$. Since $R$ and $R'$ are isomorphic, their canonical decompositions have the same number of classes, hence  $(A'\cap C)_{C\in \Mon (R)}$ has the same number of classes  as $\Mon (R)$. These number being finite, the partition $(A'\cap C)_{C\in \Mon (R)}$ coincides with $\Mon (R')$,  hence the frequency sequences $(\vert A'\cap C\vert)_{C\in \Mon (R)}$ and $(\vert C\vert)_{C\in \Mon(R)}$ are equal up to a permutation. Since $\vert A'\cap C\vert \leq \vert C\vert$ for each $C\in \Mon(R)$, these sequences must be equal. By the same token,  we obtain that the  frequency sequences $(\vert A\cap C\vert)_{C\in \Mon (R)}$ and $(\vert C\vert)_{C\in \Mon(R)}$ coincide, proving that $(2)$ holds.

 \end{proof}

The case of symmetric graphs is particularly simple:

\begin{lemma} Let $G:= (V, \mathcal E)$ be a symmetric graph. A partition  $(V_j)_{j \in J}$ of $V$ is a monomorphic decomposition of $G$ if and only if each $G_{\restriction V_j}$ is a clique or an independent set and  $G$ is the  lexicographic sum of the $G_{\restriction V_j}$'s indexed by a graph $H$ on $J$. In particular, $G$  has a finite monomorphic decomposition if and only if it is  a lexicographic sum of cliques or independent sets indexed by a finite graph. 
 \end{lemma}
\begin{proof}
Since $G$ is symmetric,  if $(V_j)_{j\in J}$  is a monomorphic decomposition of $G$ then each  $G_{\restriction V_j}$ is either a clique or an independent set and a module of $G$.   Hence, $G$ is the  lexicographic sum of the $G_{\restriction V_j}$ indexed by a graph $H$ on $J$. Conversely, if $G$ is a lexicographic sum $\sum_{j\in H}L_{j}$ of cliques or independents sets $L_{j}$ indexed by a graph $H$ then the family of $L_{j}$ forms a monomorphic decomposition of $G$. Indeed, let $F$, $F'$ be two finite subsets of $V$ such that $\vert F\cap L_j\vert = \vert F'\cap L_j\vert $ for every $j\in H$. For $j\in J$, let  $f_j$ be  any bijective map from $F\cap L_j$ onto $F'\cap L_j$, then $f:= \bigcup_{j\in J}f_j$ is an isomorphism of $G_{\restriction F}$ onto  $G_{\restriction F'}$.
\end{proof}\\

\noindent{\bf Proof of Theorem \ref{thm:lex-one-sibling}}
Let $A\subseteq V$ be such that $G$ embeds into $G\restriction A$. Our aim is to show that $G_{\restriction A}$ is isomorphic to $G$. According to $(2)$ of Proposition \ref{prop:mono-dec}, $\vert C\cap A\vert = \vert C\vert$ for each equivalence class $C$ of $\simeq_G$. For each equivalence class $C$, let $f_C$ be any bijective map from $C$ onto  $C\cap A$. Then $f:= \bigcup_{C} f_C$ is an isomorphism from $G$ onto $G_{\restriction A}$. \hfill $\Box$\\

\noindent{\bf Proof of implication $(ii) \Rightarrow (iii)$ of Theorem \ref{thm:cograph- one-sibling}.}  We argue by induction. Since the quasi order $\preceq$ definined in the proof of Theorem \ref{thm:connected-siblings} is a well quasi order, in order to prove that $\sib_{c} (G)=1$ implies that $G$ is a lexicographical sum of cliques or independent sets, we may suppose that for every graph $H$ such that $H \prec G$  and $\sib_c (H)=1$  is such a lexicographic sum.  According to Theorem \ref{label:chain/antichain}, $G$ or $G^c$ is disconnected.  Without loss of generality, we may suppose that $G$ is disconnected. If $H$ is a connected component, $H\prec G^c$. Otherwise, as in the proof of Theorem \ref{thm:connected-siblings}, $G$ has infinitely many disconnected sibling contradicting $\sib_c(G)\not =1$. We apply Proposition \ref{summary}. For each  connected component $sib_{con}(H)=1$    and,  due to the w.q.o.  of embeddability, there are only finitely many  connected  components. Let $\{H_i: i<m\}$ be the set of non-trivial connected components of $G$.  Due to the induction hypothesis, each one is of the form $\sum_{j\in  K_i}L_{ij}$, where each $L_{ij}$ is a finite cograph. Let $K:= \bigoplus_{i<m} K_{i} \oplus \{a\}$ and $L_a$ be the independent set made of the trivial components. Then, $G$ is the lexicographical sum of the $L_{ij}$ and $L_a$ over $K$. 
\hfill $\Box$

\section{Extensions} \label{section:extension}

Our result  is crude in several aspects.

 First, we think  that one can prove without (CH) that a countable cograph has one, $\aleph_0$ or $2^{\aleph_0}$ siblings. 
 
 We think that the following holds:  \label{thm:siblings-cographs2}

Let  $G$ be a countable cograph and ${\bf T}(G):= (R(G), v)$ its decomposition tree.
\begin{enumerate}

\item $\sib(G)= 2^{\aleph_{0}}$ if ${\bf T}(G)$ contains an infinite set  $A$ such that for every integer $n$ the set of  $a\in A$ such that the subtree $T_{\restriction \uparrow a}$ has cardinality at most $n$ is finite.
\item $\sib (G)= \aleph_{0}$  if ${\bf T}(G)$ has only finitely many levels and for each $a\in T$ with infinitely many successors, there is an integer $n$ which bounds the cardinality of almost all $T_{\restriction \uparrow b}$ (where $b$ is a successor of $a$), and there is some $a\in T$ with infinitely many successors $b$  such that  all  $T_{\restriction \uparrow b}$ have at least two elements. 
\item $\sib(G)= 1$ if ${\bf T}(G)$ has only finitely many levels and if some element $a$ has infinitely many successors then almost all are maximal in ${\bf T}(G)$. 
\end{enumerate}

What is needed?

$\bullet$ \label{sib(1+G)} A countable connected cograph $G$ embedding $G\oplus 1$ has $2^{\aleph_0}$ siblings. 

This will be true if we can prove that

 $\bullet$ If ${\bf T}(G)$ is well founded with an infinite chain then $\sib(G)=2^{\aleph_0}$.

\begin{problem}
If a  a countable  connected graph $G$ embeds $G\oplus 1$, is $\sib(G)= \aleph_0$ or $\sib(G)=2^{\aleph_0}$? 
\end{problem}
%The  one way infinite path $P_{\omega}$ embeds $1\oplus P_{\omega}$. It has countably many siblings. However: 
%
%
%JE PENSE QUE ON A TOUS LES ELEMNTS POUR CARACTERISER LES COGRAPHES denombrables dont le nombre de siblings est au plus denombrable. 
%
% 
 We guess that the alternative "one" or "infinite" may hold for arbitrary cographs, possibly uncountable. But, we may note that then the well quasi ordering arguments cannot be used. The collection of uncountable cographs is not w.q.o. under embeddability. Simple examples can be made with rigid chains and the comb construction. Also, there are uncountable cographs  with no proper embedding (in particular,   they have just one sibling) while countable cographs with one sibling have plenty of  proper embeddings. To illustrate, say that  a \emph{comb} is  the  sum $G$  of a dense labelled chain $C:= (I,\leq,  \ell)$ such that for each $i$,  the  label  $\ell(i):= (G_i, v(i))$  is made of a one vertex graph if $i$ is not the largest element, otherwise $G_i$ has two vertices,    and $v(i)\in \{0,1\}$. The decomposition tree of $G$ is the pair $(T, w)$ where $T$ is the tree on $I\cup I'$ with an extra element $a$ if $I$ has a largest element. The order on $I\cup I'$ extend the order on  $I$, the set $I'$ is an antichain, every element $i$ of $I$ has a unique successor $i'\in I'$, except if $i$ is maximal in $I$, in which case it has  two, namely $i'$ and $a$. The label  function $w$ is $v$ (see the example given in Figure \ref {Examples2}).  

It is easy to construct uncountable  combs with no proper sibling. A rigid chain will do, but this is unnecessary.  Indeed,  Dushik and Miller \cite{dushnik-miller} (see also Chapter 9 of
Rosenstein \cite{rosenstein}, Theorem 9.6 page 151) showed that the real line $\R$ can be decomposed into two
disjoint dense subsets $E$ and $F$ such that $g(E) \cap F \neq
\emptyset$ and $g(F) \cap E \neq
\emptyset$ for any non-identity order preserving map $g:\R
\rightarrow \R$. Thus, let $C:=(\R, \leq, \ell)$ where $\ell(i):= (G_i, \chi_F(i))$  is made of the one vertex vertex graph and  $\chi_F(i)= 1$ if $i\in F$,  and $0$  if $i\notin F$. Then the corresponding comb has no proper embedding. 

 Instead of cographs, one could consider series-parallel posets, that is posets not embedding an $"N"$ or equivalently posets whose  comparability graph is a cograph. 
More generally, let $\mathcal C$ be a  hereditary class of finite binary structures containing only finitely many indecomposable structures. According to  unpublished results of Delhomm\'e \cite{delhomme} and G.Mckay \cite{mckay},  the collection of countable binary structures $R$ such that $\age (R)\subseteq \mathcal C$ is b.q.o. (even if we add labels). There is not much difference with the case of cographs; we have just to add the case of a lexicographical  sum indexed by a finite indecomposable structure.

\section{Appendix: cographs and labelled trees}

In this section,  we  prove the existence of  a one-to-one correspondence between cographs and ramified meet-trees densely valued by $\{0,1\}$ ( cf. Theorem \ref{thm:correspondence}). This result follows from Lemma 5.1 and  Proposition 5.4 of \cite{courcelle-delhomme} as follows.  In \cite{courcelle-delhomme} a labelled tree $mdec(G)$ is constructed which consists of the one element subsets of $G$,  and the robust modules of $G$ together with strong modules which are limit modules that are maximal proper strong submodules of a robust module. If $A$ is a robust module which is not a singleton, a label is assigned according to the structure of the quotient of $A$ by the family of maximal proper strong submodules of $A$.  For cographs only two labels arise according to whether the quotient is a complete graph, or an independent set.  From this tree one defines a graph on the "leaves", by assigning an edge between distinct elements $x$ and $y$ according to the label of the robust module that $x$ and $y$ determine.   Lemma 5.1 asserts that the original graph $G$ is recovered. Proposition 5.4 of \cite{courcelle-delhomme} asserts that the labeled tree $mdec(G)$ is uniquely determined by  $rdec(G)$, the labelled tree  of robust modules,  by a process of completion, essentially adding the limit strong modules that are maximal proper submodules of a non-singleton robust module.  Since their tree $rdec(G)$ is obtained from our tree of robust submodules ordered by reversing inclusion, changing join to meet, their results apply.  This establishes the result. We think that correspondence is simple and  important enough to justify a detailed presentation.

We  first put together some general properties of modules, and the modular decomposition of a graph.  In order that the Appendix may be used for other work, we present results in more generality, rather than provide statements and proofs strictly in the context of cographs.

\subsection{Modules} 

We  recall some basic ingredients of binary structures, alias $2$-structures. Most of it can be found in  \cite {ehrenfeucht1}. Let $W$ be  a set. A   \emph{binary structure over $W$} is a pair $\M:=(V, d)$ where $d$ is a map from $V\times V$ into  $W$; its \emph{restriction} to a subset $A$ of $V$ is $M_{\restriction A}:= (A, d_{\restriction A\times A})$. The value of $d$ on $\Delta_V:= \{(x,x)\in V\times V: x\in V\}$, the diagonal of $V$, plays no role in the notions involved below, and the reader could  suppose that it is constant. 

A subset $A$ of  $V$ is a \emph{module} of $M$ if $d(x,y)=d(x,y')$  and $d(y,x)=d(y',x)$ for every $x\in V\setminus A, y,y'\in A$. 
(other names are autonomous sets, or intervals). 

The whole set, the empty set  and the 
singletons are modules. These are the  \emph{trivial} modules.  A binary structure whose modules are trivial is 
 \emph{indecomposable}. If moreover it has more than two vertices it is  \emph{prime}.
   
 We recall  the basic and well-known properties of modules, under the form given in \cite{courcelle-delhomme}. 
\begin{lemma}\label{lem:intersection-union} Let $\M:=(V,d)$ be  a binary structure.
Then:
\begin{enumerate}
\item The intersection of a non-empty set of modules is a module (possibly empty). 
\item  The union of two modules that meet is a module, and more generally, the union of a set of modules is a module as
soon as the meeting relation on that set is connected. 
\item For two modules $A$ and $B$, if $B\setminus A$ is non-empty, then $A\setminus B$ is a module. 
\end{enumerate}
\end{lemma}

 A module is \emph{strong} if it is either comparable w.r.t. inclusion or disjoint from  every other module.

\begin{lemma} \label{strongmodule}
\begin{enumerate} \item The intersection of any set of strong modules is empty or is a strong module. 
\item The union of any non-empty directed set of strong modules is a strong module.
\end{enumerate}
\end{lemma}

Let $A$ be a subset of $V$. Let  $S_{\M}(A)$ be the intersection of strong modules  of $M$ which contain $A$. We write
$S_{\M}(x,y)$ instead of $S_{\M}(\{x,y\})$. According to $(1)$ of Lemma \ref {strongmodule}, $S_{\M}(A)$ is a strong module. 

According to Courcelle, Delhomm\'e  2008 \cite{courcelle-delhomme}, a module $A$  is \emph{robust} if it  is either a singleton or the least strong module containing two distinct vertices; that is there are $x,y\in A$ such that $A$ is strong and every strong module containing $x$ and $y$ contains $A$. Alternatively, $A=S_{\M}(x,y)$ for some $x,y \in V$. 
 \begin{example}\label{prime-robust}
 If $A$ is a  module and $\M{\restriction A}$ is prime then $A$ is robust.
  \end{example}
 
% \begin{proof}
% We check that $A$ is a strong module. Suppose that there is some module  $B$ which meets $A$ properly. According to $(1)$ and $(3)$ of Lemma \ref{lem:intersection-union}, $A\cap B$ and $A\setminus B$ are  modules; since $\M{\restriction A}$ is indecomposable, they must be trivial, hence $A$ has two elements, thus it is not prime, a contradiction. Since  $\M{\restriction A}$ is prime it is indecomposable, hence if $x, y$ are two distinct elements of $A$, $S_{\M}(x,y)=A$ and  $A$ is robust. \end{proof}
%
%
% 
% \begin{lemma}\label{lem:strong-notrobust}A  subset $A$ of $V$ is a strong module which is not robust iff $A$ is the union of a  family of robust modules which are distinct from $A$ and which is totally ordered w.r.t. inclusion.
%  \end{lemma}
%  \begin{proof}Suppose that $A$ is the union of a  family of robust modules distinct from $A$ which is totally ordered w.r.t. inclusion. According to $(2)$ of Lemma \ref{strongmodule}, $A$ is strong. Furthermore, if $x,y\in A$ there is a member of the family containing $x,y$,  hence $A$ cannot be robust. Conversely, pick any  $x\in A$. Then  the collection of robust modules included into $A$ forms a chain and their union is $A$. 
% \end{proof}
% 
%   
   
\begin{definition} Let $A$ be a strong module. Let $x,y\in A$. We  set $x\equiv_A y$ if  either $x=y$ or there is a strong module containing $x$ and $y$ and properly contained in  $A$. 
\end{definition}

 \begin{lemma}\label{decomp-strong}The relation $\equiv_A$ is an equivalence relation on $A$ whose equivalence classes are strong modules. If $A$ has more than one element then there are at least two classes iff $A$ is robust. Furthermore, these classes are the maximal strong modules  properly included in $A$
 \end{lemma}
 \begin{proof} Since strong modules are disjoint or comparable, $\equiv_A$  is an equivalence relation. Suppose that $A$ has at least two elements. If $A$ is robust then there are two distinct elements such that $A=S_{\M}(x,y)$, hence $x\not \equiv_Ay$ proving that there are at least  two equivalence classes.  Conversely, if there are at least two classes, then pick $x$ in one and $y$ in another; since $x\not \equiv_Ay$ we have $S_{\M}(x,y)=A$. 

 Let $I$ be  an equivalence class. Suppose that there is some strong module $F$ such that $I \subseteq F\subset A$ with  $F$ strong. Then all members of  $F$ are $A$-equivalent. Since $I$ is an equivalence class,  $I=F$. This shows that  if $I$ is strong  then  it is maximal. The fact that  $I$ is strong follows from (2) of Lemma \ref{lem:intersection-union}, but can be obtained directly as follows. Let  $J$ be a module  which  meets  $I$.  We claim $J$ is comparable to $I$. Since $A$ is strong, we may suppose that $J\subseteq A$ and $J$ is incomparable to $I$. Let $a\in I\cap J$ et $b\in I\setminus J$.  Then $S_{\M}(\{a,b\})$, the least strong module containing $a$ and $b$,  which is necessarily contained in $I$, intersects  $J$ properly, contradicting the fact that it is strong. 
\end{proof}

We call  \emph{components} of $A$ the equivalence classes of the relation $\equiv_A$. 
Except if $A$ is finite, the components  need not   be robust.

 A notion equivalent to the notion of robust module was previously introduced by Kelly \cite{kelly} see also \cite {boussairi-al}. A module  $I$ is  \emph{non-limit} if it is strong and contains a non-empty strong module $J$ which is maximal among those contained in $I$ and distinct from $I$.   
 
 \begin{proposition}
Let $A$ be a  subset of $V$. Then   $A$ is a robust module with at least two elements  iff  $A$ is a non-limit module.   
\end{proposition}

\begin{proof}
Suppose that $A$ is  a non-limit module.  Let  $ x\in I\subset A$ with  $I$  maximal among the strong modules contained in $A$ and distinct from $A$. Let  $y \in A\setminus I$. Let  $A':=S_{\M}(x,y)$. Since $I$ is strong, $A'$ is a strong module properly containing  $I$. Due to the choice of  $I$,  it is equal to $A$ hence $A$ is robust.
Conversely, suppose that $A$ is robust with at least two elements.  Then the components of $A$, as defined above,  are the maximal non-empty strong modules properly contained  in $A$ and in particular $A$ is non-limit.  
\end{proof}

 If $A$ is a robust module with at least two elements of a binary structure $\M:= (V, d)$ then for two distinct components $I, J$ of $A$, the values $d(x,y)$ for $x\in I$ and $y\in J$ depends only upon $I$ and $J$. Hence, the binary structure on $A$ induces  a binary structure $\M/{\equiv_{A}}$ on the set $A/{\equiv_{A}}$ of components of $A$, called the \emph{Gallai quotient of $A$}, and $\M_{\restriction A}$ is the lexicographical sum of the $\M_{\restriction I}$'s indexed by $\M/{\equiv_{A}}$. This quotient $\M_{\restriction A/{\equiv A}}$ has a special structure: its  strong modules are trivial,  there are only  the empty set, the whole set ${A/\equiv}$ and the singletons.

 The  central result  of the decomposition theory of binary structures describes the structure of the Gallai quotient.  
It is  due to Gallai \cite{gallai} for finite graphs, to Ehrenfeucht and Rozenberg \cite{ehrenfeucht} for finite binary structures and to Harju and Rozenberg \cite {harju-rozenberg} for infinite binary structures (see also \cite{courcelle-delhomme} Corollary 4.4.)

 \begin{theorem} The strong modules of  a binary structure $\M$ are trivial iff either $\M$ is prime, or \emph{constant}, that is $d(x,y)= \alpha$ for all $x\not =y\in V$, or \emph{linear}, that is there are $\alpha \not =\beta$ such that $\{(x, y)\in V\times V: x\not =y, d(x,y)= \alpha \}$ is a linear (strict) order  and $\{(x, y)\in V\times V: x \not =y, d(x,y)= \beta \}$ is the opposite order.
\end{theorem}

We say that the \emph{type} of $A$, $t(A)$,  is prime  if $\M_{\restriction A/{\equiv}}$ is prime, otherwise its type is $\alpha$ if $\M_{\restriction A/{\equiv A}}$ is constant, and $\{\alpha, \beta\}$ if it is linear.

In the case of directed graphs, this yields:  

\begin{theorem}\label{thm:6.8} If $\M$ is a directed graph, every  robust module with at least two elements is the lexicographic sum of its components and the quotient is either a clique or an independent set or a chain or a prime graph. 
\end{theorem}

%The following result gives a description of non-strong modules. 
%
%\begin{proposition} A  module $A$ is not strong iff  $S_{\M}(A)$ is robust and distinct from $A$, the Gallai quotient of $A$ is not prime,  $A$ is an union of components which form a proper autonomous subset of this quotient.
% \end{proposition}
% \begin{proof}
% If $A$ is not strong then $S_{\M}(A)\not =A$ (and conversely). Next $S_{\M}(A)$ is  robust. Otherwise, since  $S_{\M}(A)$ is strong then,  according to Lemma \ref {lem:strong-notrobust},  $S_{\M}(A)$  is the union of a  family of robust modules distinct from $S_{\M}(A)$ which is totally ordered w.r.t. inclusion. Pick $x\in A$; since they are strong, the members of the   family which contains $x$ are comparable to  $A$;  since they cannot contain $A$ they are included into $A$, hence $A$ is the union of a totally ordered family of strong modules, hence $A$ is strong, which is not the case. Since $S_{\M}(A)$ is  robust and has at least two elements, it has at least two components.  Since the  components are strong, none contain $A$ and $A$ is an union of at least two components. Since $A$ is a module of $S_{\M}(A)$, the set of these components  is a module of the Gallai quotient. The converse is straightforward. 
% \end{proof}
% 
Let $X, Y$ be two disjoint non-empty subsets of $\M$. If $X, Y$ are two modules, the value $d(x,y)$ where $x\in X$ and  $y\in Y$ is independent of $X$ and $Y$, we will denote it by $d(X,Y)$. 

\begin{lemma}\label{lem:ad-hoc}
Let $X,Y$ be two disjoint non-empty modules of $\M$. If $Y$ is robust, non-trivial and $\{d(X,Y), d(Y,X)\} = t(Y)$ then $X\cup Y$ is not a module. \end{lemma}
\begin{proof} Let  $\alpha:= d(X,Y)$ and $\beta:= d(Y,X)$. By hypothesis, we have $t(Y)= \{\alpha, \beta\}$. Let $L:= \{(p,q)\in Y/{\equiv_Y}: p\not = q \; \text{and} \; d(p,q)= \alpha\}$. If $\alpha \not = \beta$ this is a linear order, otherwise this is a complete graph or an independent set.  Since the quotient has at least two elements, we may divide it into two non-empty subsets, which in the case $\alpha\not = \beta$ are an   initial segment $I$ and a final segment $F$ w.r.t to this order. Let $Y'$ be the union of components of $Y$ which belong to $I$.  We claim that $X\cup Y'$ is a module whenever $X\cup Y$ is a module. Indeed, suppose that $X\cup Y$ is a module. Let  $x, x'\in X\cup Y'$ and $y\in V\setminus (X\cup Y')$.  We check that $d(x,y)=d(x',y)$ and $d(y,x)=d(y,x')$. If $y\not \in X\cup Y$ this holds since $X\cup Y$ is a module. Thus,  we may suppose $y\in Y\setminus Y'$. If $x,x'\in Y'$ this holds because due to our choice, $Y'$ is a module of $Y$. If $x,x'\in X$ this holds because $X$ is a module. Hence,  we may suppose $x\in X, x'\in Y'$. In this case, we have $d(x, y)=d(X, Y)= \alpha$ and $d(x', y)= \alpha$, hence $d(x,y)=d(x',y)$; similarly $d(y,x)=d(y,x')$, proving that $X\cup Y'$ is a module. But, this is impossible since it meets $Y$ properly   and $Y$ is strong. 
 \end{proof}
\begin{proposition}\label{lem:density} If two  robust modules $A$ and $B$ with at least two elements and such that  $B\subset A$ have the same non prime type $\{\alpha, \beta\}$ then there is a robust module $C$ with $B\subset C\subset A$ whose type is  distinct from  the type of $\{\alpha, \beta\}$.

\end{proposition}

\begin{proof}
Suppose that this is not the case. That is $t(C)=\{\alpha, \beta\}$ for every  robust module $C$ with $B\subset C\subset A$. 
\begin{claim}\label {claimmodule1}
Let $x\in A\setminus B$ and $y,y'\in B$. Then $d(x,y)=d(x,y') \in \{\alpha, \beta \}$. 
\end{claim}
\noindent {\bf Proof of Claim \ref{claimmodule1}.}
Let $C$ be the least strong module containing $x, y$ and $y'$. We have $C= S_{\M} (x,y)=S_{\M} (x,y')$, hence $B\subseteq C \subseteq A$ thus $t(C)= \{\alpha, \beta\}$. Since $x$ and $y$ belong to two different components of $C$, $d(x,y)\in t(C)$ and $d(y,x)\in t(C)$; similarly $d(x,y'), d(y'x)\in t(C)$. Since $B$ is a module,  $d(x,y)=d(x,y')$ and $d(y,x)=d(y',x)$. The claim follows. 
\hfill $\Box$

Let $d(x, B):= d(x, y)$ where $y\in B$. Let $\gamma\in t(A)$ and $X_{\gamma}:= \{x\in A\setminus B: d(x, B)= \gamma\}$. 
Then, according to Claim \ref{claimmodule1},  $A= X_{\alpha}\cup X_{\beta} \cup B$.
\begin{claim}\label {claimmodule2}
$X_{\alpha}$ and $X_{\beta}$ are modules of $\M$. 
\end{claim}
\noindent {\bf Proof of Claim \ref{claimmodule2}.} Let $x, x'\in X_{\alpha}$ and $y\in V\setminus X_{\alpha}$. If $y\not \in A$ then since $A$ is a module of $\M$ we have $d(x,y)=d(x',y)$ and $d(y,x)=d(y,x')$ as required. If $y\in A$ then since $A= X_{\alpha}\cup X_{\beta} \cup B$, 
either $y\in B$ or $y\in X_{\beta}$ in which case  $\alpha\not = \beta$. If $y\in B$ then by definition of $X_{\alpha}$ we have $d(x,y)=d(x',y)= \alpha$ and hence $d(y,x)=d(y,x')=\beta$. If $y\in X_{\beta}$, pick $z\in B$. Since $x\in X_{\alpha}$ and $y\in X_{\beta}$ we have $d(x, z)=\alpha$ and $d(z, y)=\alpha$;  since $\alpha\not = \beta$, the Gallai quotient of $A$ is linear, hence $d(x, y)= \alpha$; similarly, $d(x',y)=\alpha$. Thus $X_{\alpha}$ is a module. The same holds for $X_{\beta}$.

Since $A= X_{\alpha}\cup X_{\beta} \cup B$ one of the sets $X_{\alpha}, X_{\beta}$ is non-empty. Suppose that this is $X_{\alpha}$.

\begin{claim}\label {claimmodule3} $X_{\alpha}\cup B$ is a module. 
\end{claim}
\noindent {\bf Proof of Claim \ref{claimmodule3}.}
If $\alpha=\beta$, $X_{\alpha}\cup B= A$ and  there is nothing to prove.  
Suppose $\alpha\not = \beta$. Let  $x, x'\in X_{\alpha}\cup B$ and $y\in V\setminus (X_{\alpha}\cup B)$.  We check that $d(x,y)=d(x',y)$ and $d(y,x)=d(y,x')$. If $y\not \in A$ this holds since $A$ is a module. Thus,  we may suppose $y\in A\setminus (X_{\alpha}\cup B)$, that is  $y\in X_{\beta}$. If $x, x'\in X_{\alpha}$ or $x,x'\in B$ these equalities holds since $X_{\alpha}$ and $B$ are modules.  Thus we may suppose $x \in  X_{\alpha}$ and $x'\in B$. In this case, we have $d(x, x')=d(x',y)= \alpha$ and since $L$ is linear,  $d(x, y)=\alpha=d(x',y)$;  by the same argument  we also have  $d(y, x)=d(y,x')=\beta$. This proves our claim.

\begin{claim}\label {claimmodule4} There is some non-empty proper subset $D$ of $B$ such that $X:=X_{\alpha}\cup D$ is a module. 
\end{claim}
\noindent {\bf Proof of Claim \ref{claimmodule4}.}
Case 1. $\alpha=\beta$. Let $D$ be a component of $B$. Since $B$ is non-trivial, $D$ is a proper subset of $B$. It is easy to check that  $X: = X_{\alpha}\cup D$ is a module. For that, pick $x,x'\in X$ and $y \in V\setminus X$. If $x\in X_{\alpha}, x'\in D$ we have $d(x, x')= \alpha$, hence, for every $y\in A$ we have $d(x, y)=d(x',y)=\alpha$; the equality $d(x, y)=d(x',y)$ holds trivially in all other cases. 

Case 2. $\alpha\not = \beta$. In this case, the Gallai quotient of $B$ is linear. The set $L:= \{(p,q)\in B/{\equiv_B}: p\not = q \; \text{and} \; d(p,q)= \alpha\}$ is a linear order. Since the quotient has at least two elements, we may divide it into a non-empty initial segment and a non-empty final segment w.r.t to this order. Let $D$ be the union of components of $B$ which belong to such an initial segment. As above one,  can check that $X$ is a module. 
\hfill $\Box$

This claim contradicts Lemma \ref{lem:ad-hoc}. 

\end{proof}
% 
%\begin{theorem}\label{prop:modules} A non-empty subset  $A$ of a binary structure $\M$ is a module  iff either  $A$ is a robust module or   $A$ is the union of a  family of robust modules distinct from $A$ which is totally ordered w.r.t. inclusion or $A$ 
%$S_{\M}(A)$ 
%$A$ is the union of at least two components  of a robust module containing at least two components. \end{theorem} 
%\begin{proof}
%Let A be a module  of $\M$. 
%If $S_{\M}(A)=A$ then trivially $A$ is strong. 
%\begin{claim}
%If $A\not =S_{\M} (A)$ then $S_{\M}(A)$ is robust, contains at least three classes and $A$ is an union of at least two classes.  
%\end{claim}
%\begin{subclaim} If a module  $X$ is strong and not robust it is an  union of a chain of robust modules
%\end{subclaim}
%
%
%From this subclaim, $S_{\M}(A)$ is robust. \end{proof}
%

\subsection{Decomposition tree of cographs}

The presentation followed below is equivalent to that employed by Courcelle and Delhomme \cite{courcelle-delhomme}, except they use inclusion instead of reverse inclusion and find a join-lattice rather than a meet-lattice as we do.  

Once ordered by the reverse of inclusion, the collection of strong modules forms a tree. We prefer to consider a refinement of this tree made of robust modules. The collection of robust modules of a binary structure $\M:= (E, d)$, once ordered by the reverse of  inclusion, forms  a tree,  \emph{the decomposition tree of the $2$-structure}, see \cite{courcelle-delhomme} \cite{thomasse} for some use of this tree. We describe this tree in the case of cographs.

Let $P$ be a poset. We recall that $P$ is a  \emph{forest} if for every element $x\in P$ the initial segment $\downarrow x:= \{y\in P: y\leq x\}$ is a chain; this is a \emph{tree} if in addition every pair of elements has a lower bound. We say that $P$ is a \emph{meet-lattice} if every pair of elements $x, y \in P$ has a \emph{meet} that we denote by $x\wedge y$ (and which is the largest lower bound of $x$ and $y$). 

Let $T$ be a meet-tree. We observe that if an element $x$ of $T$ is the meet of a finite set $X$ of the maximal elements of T, denoted $\Max(T)$, then $x$ is the meet of a subset $X'$ of $X$ with at most two elements. 

We say that a meet-tree $T$ is \emph{ramified} if every element of $T$ is the meet of a finite set of maximal elements of $T$.

Let $T$ be a ramified meet-tree and $T':= T\setminus Max(T)$.  A $\{0,1\}$-valuation is a map $v: T' \rightarrow \{0,1\}$. The valuation is \emph{dense} if for every $a<b$ in $T'$ there is some $c$ with $a<c\leq b$ such that $v(c)\not = v(a)$. 

\begin{lemma}

Let $(T, v)$  be a densely valued ramified meet tree. Let $G:={\bf G}(T)$ be the graph with vertex set $V:= Max(T)$, two vertices $x$ and $y$ being joined by an edge if $v(x\wedge y)= 1$. Then $G$ is a cograph and ${\bf T}(G):= (R(G), v_G)$ is isomorphic to $(T, v)$. 

\end{lemma}

\begin{proof}
Let $x,y\in V$.  Let $S_G(x,y)$ be the least strong module of $G$ containing $x$ and $y$, let $a:= x\wedge y$ and $B(x,y):= (\uparrow a) \cap V:= \{z\in V: z\geq x\wedge y\}$.

We prove that the following equality holds.

 \begin{equation}\label{eq:ball} S_G(x,y)=B(x,y). 
\end{equation}

\begin{equation}\label{eq:val}
v_G(S_G(x,y))= v(x\wedge y). 
\end{equation}

The fact that  $(R(G), v_G)$ and   $(T, v)$ are isomorphic follows.

\begin{claim} \label{claim:ball-module} $Z:= B(x,y)$ is a strong module of $G$.
\end{claim}

\noindent{\bf Proof of Claim \ref{claim:ball-module}.}
 We prove first that $Z$ is a module containing $x$ and $y$. Let $t\in V\setminus Z$. Let $a':= x\wedge y\wedge t$. We claim that $z\wedge t=a'$ for every $z\in Z$. Since $G(z,t)= c(z\wedge t)=c(a')$, if this   equality holds,  $G(z,t)$ is independent of $z$,  hence $Z$ is a module of $G$. To prove that this equality holds, let $z\in Z$. By definition, we have $x\wedge y\leq z$. From this inequality, we get  $a'=x\wedge y\wedge t\leq z\wedge t$. Since $x\wedge y\leq z$ and $z\wedge t\leq z$ and $T$ is a tree, $x\wedge y$ and $z\wedge t$ are comparable. We cannot have $x\wedge y \leq z\wedge t$ otherwise we would have $x\wedge y \leq t$ contradicting $t\not \in Z$. Thus,  we have $z\wedge t\leq x\wedge y$. Since $z\wedge t \leq t$ and $a'= x\wedge y\wedge t$ it follows that $z\wedge t \leq a'$. With the inequality $a'\leq z\wedge t$ obtained above, this gives $z\wedge t=a'$ as claimed. 
 
 Now we prove that $Z$ is a strong module. Suppose not. Let $I$ be a module which intersects  $Z$ properly. Let $x'\in Z\setminus I$, $y'\in  Z\cap I$ and $t\in I \setminus Z$. Let $Z':= B(x', y')$, $b:= x'\wedge y'$ and $b':= b\wedge t$. We have $Z'\subseteq Z$, hence $t \not\in Z'$. The set  $Z'$ is a module and $x'\wedge t=y'\wedge t$. Since $I$ is a module,  $G(x', y')= G(x', t)$ hence $v(b)=v(b')$. We have $b'< b$ hence according to Lemma \ref{lem:density} there is some element $c$ with $b'<c<b$ such that $v(b')\not = v(c)$. Since $T$ is ramified, there are two elements $x'',y''\in V$ such that $x''\wedge y''=c$. Let $Z'':= B(x'',y'')$. Then $Z\subset Z''$. Necessarily, $x''$ or $y''$ is not in $Z'$. Suppose that this is $y''$. In this case,  we have $x'\wedge y''=y'\wedge y''= x''\wedge y''= c$. If $y''\not \in I$  then, since $I$ is a module, we must have $v(y'\wedge y'')= v(t\wedge y'')$. This is impossible since  $y'\wedge y''=c$,  $t\wedge y''=b'$ and  $v(b')\not = v(c)$. Suppose that $y''\in I$ then since $I$ is a module we must have $v(x'\wedge y')= v(x'\wedge y'')$ which is impossible since $x'\wedge y'= b$,  $x'\wedge y''=c$ and  $v(b)= v(b')\not = v(c)$. Consequently, $I$ cannot intersect  $Z$ properly, proving that $Z$ is strong. 
 
 \hfill $\Box$

On $Z$ define the following  binary relation $\equiv_Z$:
\begin{equation}
p\equiv_Z q  \; \text{if}\;  p,q\in Z\; \text{and}\;  p=q\;  \text{or}\; p\wedge q\not = a. 
\end{equation}

\begin{claim}\label{claim:equivalence}
The relation $\equiv_Z$ is an equivalence relation whose blocks are strong modules. If $Z$ is not a singleton then there are at least two blocks and $G_{\restriction Z}$ is a lexicographical sum on these blocks indexed by a clique or an independent set.   
\end{claim}

{\bf Proof of Claim \ref{claim:equivalence}.} 
The relation $\equiv_Z$ is clearly reflexive and symmetric.  We check that it is transitive. 
Let $p,q,r\in Z$ such that $p\equiv q$ and $q\equiv r$. We may suppose that $p,q, r$ are pairwise distinct, hence  $p\wedge q>a$ and $q\wedge r>a$. Since $p\wedge q\leq q$ and $q\wedge r\leq q$ and $T$ is a tree, $p\wedge q\leq q$ and $q\wedge r\leq q$ are comparable, hence $a<Min \{p\wedge q,  q\wedge r\}\leq p\wedge r$ proving that $p \equiv_Z r$  is transitive. If $Z$ is not a singleton, then $x\not =y$ hence  $x\not \equiv_Z y$ and $x\wedge y =a$, hence,  there are at least  two blocks. Let $I$ be  a block. Pick $p\in I$, then $I= \bigcup_{q\in I} B(p, q)$. Since, according to Claim \ref{claim:ball-module}, each $B(p,q)$ is a strong module, $I$ is a strong module (as a union of strong modules with a common vertex, see Lemma \ref{strongmodule}). If $I$ and $J$ are two distinct blocks, let $p\in I$ and $q\in J$.  Since $p\wedge q=a$, $G(p,q)= v(a)$ hence $G_{\restriction Z}$ is a lexicographical sum on the  blocks indexed by a clique if $v(a)=1$ and indexed by an independent set if $v(a)=0$. 
\hfill $\Box$\\

\noindent{\bf Proof of Equation (\ref{eq:ball}).} If $x=y$, $B(x,y)=S_G(x,y)=\{x\}$. Suppose $x\not =y$. We have $S_G(x,y)\subseteq B(x,y)$. Indeed, by definition $S_G(x,y)$ is the least  strong module containing $x$ and $y$. According to Claim \ref{claim:ball-module},  $B(x,y)$ is a strong module. Since it contains  $x$ and $y$ it contains $S_G(x,y)$. Let us prove the converse. Set $Z:= B(x,y)$. Since each block of $\equiv_Z$ is a  module,   $S_G(x,y)$,  which is a strong module,  must be comparable to every block that it meets. Since it meets the block containing $x$ and the block containing  $y$,  it contains these two blocks. Due to that fact, it contains every other block that it meets. Hence $S_G(x,y)$ is a union of  at least two blocks. Since this is a strong module, it induces a strong module on the quotient. This quotient being a clique or an independent set, this strong module must be either  a singleton, that is $S_G(x,y)$  is a block, which is not the case, or the the whole set, in which case $S_G(x,y)= Z$ as claimed. \hfill $\Box$\\

\noindent{\bf Proof of Equation (\ref{eq:val}).} According to Equation (\ref{eq:ball}), we have $S_G(x,y)=B(x,y)$. We claim that the equivalence relations $\equiv_Z$ and $\equiv_a$ coincide. Indeed, let $p,q \in Z$. We have $S_G(p,q)= B(p,q)$. Hence, $p\equiv_ Zq$ amounting to $S_G(p,q)\not = Z$ is equivalent to $B(p,q)\not =Z$ amounting to $p\equiv_a q$.  It follows that  $v_G(S_G(x,y))= v(a)$ as claimed. 

\end{proof}

\begin{lemma}
Let $G$ be a cograph, $R(G)$ be the set of robust modules of $G$ ordered by reverse inclusion,  $R_{\geq 2}(G):=\{A\in R(G): \vert A\vert \geq 2 \}$ and $v_G:R_{\geq 2}(G)\rightarrow \{0,1\}$ defined by setting $v_G(A):= 0$ if the Gallai quotient of $A$ is an independent set and $v_G(A):=1$ otherwise. Then  ${\bf  T}(G):= (R(G), v_G)$ is a dense valued meet tree  and the graph ${\bf G}$ associated with ${\bf T} (G)$ is $G$. 

\end{lemma}

\begin{proof}

We prove successively:

$(1)$\emph{ The set $R(G)$ of robust modules of $G$ ordered by reverse inclusion  is a   meet-tree;   the maximal elements of 
$R(G)$ are the singletons of $V(G)$,  and thus $R(G)$ is ramified}. 

Since robust modules are strong, they are disjoint or comparable hence $R(G)$ is  a forest. Let $A, B \in T(G)$. Pick $x\in A$ and $y\in B$. Then $S_G(x,y)$ is the meet of $A$ and $B$ in $R(G)$ hence $R(G)$ is  a meet-tree. For each $x\in V:= V(G)$, $\{x\}\in R(G)$. Let $A\in R(G)$;  since $A$ contains $x, y$ such that  $A=S_G(x,y)$,  we have $A= \{x\}\wedge \{y\}$ in $R(G)$, hence $R(G)$ is ramified. 

$(2)$ \emph{$v_G$ is a dense valuation}. 

Since $G$ is a cograph, no  strong module can be prime; since $G$ is undirected  the type of a robust module is $\{0\}$ or $\{1\}$. Hence the valuation of a robust module is essentially its  type. The density property follows from 
Lemma \ref{lem:density}. 

 $(3)$ \emph{The graph ${\bf G}$ associated with ${\bf T} (G)$ is $G$}. 
 
 Let $x,y$ be two distinct vertices of  $G$ and $A$  be the least robust module containing $x$ and $y$. Since $G$ is a cograph, Theorem \ref{thm:6.8} asserts that $G_{\restriction A}$ is a lexicographic sum of its components and the quotient is a clique or an independent. Thus, $\{x,y\}$ is an edge iff $v_G(A)=1$. 
\end{proof}

From these two lemmas  we deduce:

\begin{theorem}\label{thm:correspondence}There is a one-to-one correspondence between cographs and ramified meet-trees densely valued  by $\{0, 1\}$. \end{theorem}

\end{document}